\documentclass{book}
\pdfoutput=1
\listparindent\parindent
\usepackage{amsmath,amsthm,graphicx,tikz,url,stmaryrd,enumitem,txfonts,microtype}
\SetSymbolFont{stmry}{bold}{U}{stmry}{m}{n}

\theoremstyle{definition}
\newtheorem{definition}{Definition}[section]
\newtheorem{example}[definition]{Example}
\newtheorem{question}[definition]{Question}
\theoremstyle{plain}
\newtheorem{conjecture}[definition]{Conjecture}
\newtheorem{theorem}[definition]{Theorem}
\newtheorem{lemma}[definition]{Lemma}
\newtheorem{corollary}[definition]{Corollary}

\usepackage[bf,small,center]{titlesec}
\titlespacing*{\chapter}{0pt}{50pt}{40pt}
\titlespacing*{\section}{0pt}{4.5ex plus 1ex minus .2ex}{2.3ex plus .2ex}

\setlist[enumerate]{label=\rm{(\roman*)},topsep=2pt,itemsep=0pt,leftmargin=7mm, parsep=0pt plus 2pt,partopsep=0pt plus 1pt}

\makeatletter
\renewenvironment{thebibliography}[1]
     {\section*{References}
      \@mkboth{\MakeUppercase\bibname}{\MakeUppercase\bibname}%
      \list{\@biblabel{\@arabic\c@enumiv}}%
           {\settowidth\labelwidth{\@biblabel{#1}}%
            \leftmargin\labelwidth
            \advance\leftmargin\labelsep
            \@openbib@code
            \usecounter{enumiv}%
            \let\p@enumiv\@empty
            \renewcommand\theenumiv{\@arabic\c@enumiv}}%
      \sloppy
      \clubpenalty4000
      \@clubpenalty \clubpenalty
      \widowpenalty4000%
      \sfcode`\.\@m}
     {\def\@noitemerr
       {\@latex@warning{Empty `thebibliography' environment}}%
      \endlist}
\makeatother

\renewenvironment{thebibliography}[1]
{ \begin{oldthebibliography}{#1}
  
  \setlength{\parskip}{0pt}
  \setlength{\itemsep}{1pt}
  \bgroup\small }
{ \egroup \end{oldthebibliography} }

\numberwithin{figure}{section}
\numberwithin{equation}{section}

\usepackage[symbol,hang,norule]{footmisc}
\makeatletter
\g@addto@macro\normalsize{%
  \setlength\abovedisplayskip{0.5\baselineskip plus 0.4\baselineskip}
  \setlength\belowdisplayskip{0.5\baselineskip plus 0.4\baselineskip}
  \setlength\abovedisplayshortskip{-0.3\baselineskip}
  \setlength\belowdisplayshortskip{0.5\baselineskip plus 0.4\baselineskip}
}
\makeatother
\renewcommand{\:}{\colon}
\renewcommand{\leq}{\leqslant}
\renewcommand{\geq}{\geqslant}
\newcommand{\leqn}{\trianglelefteqslant}

\newcommand{\<}{\langle}
\renewcommand{\>}{\rangle}
\renewcommand{\mod}[1]{\mathrm{ \ } (\mathrm{mod\ } #1)}
\newcommand{\Aut}{\mathrm{Aut}}
\newcommand{\Homeo}{\mathrm{Homeo}}
\newcommand{\Sym}{\mathrm{Sym}}
\newcommand{\Alt}{\mathrm{Alt}}
\newcommand{\GL}{\mathrm{GL}}
\newcommand{\PSL}{\mathrm{PSL}}
\newcommand{\Sp}{\mathrm{Sp}}
\newcommand{\e}{\varepsilon}
\newcommand{\p}{\varphi}
\newcommand{\Nat}{\mathbb{N}}
\newcommand{\Int}{\mathbb{Z}}
\newcommand{\F}{\mathbb{F}}
\renewcommand{\S}{\mathbb{S}}
\newcommand{\M}{\mathcal{M}}
\newcommand{\X}{\mathcal{X}}
\newcommand{\g}{\mathfrak{g}}
\newcommand{\C}{\mathfrak{C}}
\newcommand{\Sm}[1]{S\!_{#1}}

\begin{document}

\chapter*{The spread of finite and infinite groups \\ {\normalfont\large Scott Harper\footnotemark}}

\footnotetext[1]{School of Mathematics and Statistics, University of St Andrews, St Andrews, KY16 9SS, UK \newline \url{scott.harper@st-andrews.ac.uk}}

\section*{Abstract}\trivlist\item[]
It is well known that every finite simple group has a generating pair. Moreover, Guralnick and Kantor proved that every finite simple group has the stronger property, known as $\frac{3}{2}$-generation, that every nontrivial element is contained in a generating pair. Much more recently, this result has been generalised in three different directions, which form the basis of this survey article. First, we look at some stronger forms of $\frac{3}{2}$-generation that the finite simple groups satisfy, which are described in terms of spread and uniform domination. Next, we discuss the recent classification of the finite $\frac{3}{2}$-generated groups. Finally, we turn our attention to infinite groups, focusing on the recent discovery that the finitely presented simple groups of Thompson are also $\frac{3}{2}$-generated, as are many of their generalisations. Throughout the article we pose open questions in this area, and we highlight connections with other areas of group theory.
\endtrivlist\addvspace{26pt}

\section{Introduction} \label{s:intro}

Every finite simple group can be generated by two elements. This well-known result was proved for most finite simple groups by Steinberg in 1962 \cite{ref:Steinberg62} and completed via the Classification of Finite Simple Groups (see \cite{ref:AschbacherGuralnick84}). Much more is now known about generating pairs for finite simple groups. For instance, for any nonabelian finite simple group $G$, almost all pairs of elements generate $G$ \cite{ref:KantorLubotzky90,ref:LiebeckShalev95}, $G$ has an invariable generating pair \cite{ref:GuralnickMalle12JLMS,ref:KantorLubotzkyShalev11}, and, with only finitely many exceptions, $G$ can be generated by a pair of elements where one has order $2$ and the other has order either $3$ or $5$ \cite{ref:LiebeckShalev96Ann,ref:LubeckMalle99}. 

The particular generation property of finite simple groups that this survey focuses on was established by Guralnick and Kantor \cite{ref:GuralnickKantor00} and independently by Stein \cite{ref:Stein98}. They proved that if $G$ is a finite simple group, then every nontrivial element of $G$ is contained in a generating pair. Groups with this property are said to be \emph{$\frac{3}{2}$-generated}. We will survey the recent work (mostly from the past five years) that addresses natural questions arising from this theorem.

Section~\ref{s:finite} focuses on finite groups and considers recent progress towards answering two natural questions. Do finite simple groups satisfy stronger versions of $\frac{3}{2}$-generation? Which other finite groups are $\frac{3}{2}$-generated? Regarding the first, in Sections~\ref{ss:finite_spread} and~\ref{ss:finite_udn}, we will meet two strong versions of $\frac{3}{2}$-generation, namely (uniform) spread and total/uniform domination. Regarding the second, Section~\ref{ss:finite_bgh} presents the recent classification of the finite $\frac{3}{2}$-generated groups established by Burness, Guralnick and Harper in 2021 \cite{ref:BurnessGuralnickHarper21}. All these ideas are brought together as we discuss the generating graph in Section~\ref{ss:finite_graph}. Section~\ref{ss:finite_app} rounds off the first half by highlighting applications of spread to word maps, the product replacement graph and the soluble radical of a group.

Section~\ref{s:infinite} focuses on infinite groups and, in particular, whether any results on the $\frac{3}{2}$-generation of finite groups extend to the realm of infinite groups. After discussing this in general terms in Sections~\ref{ss:infinite_intro} and~\ref{ss:infinite_soluble}, our focus shifts to the finitely presented infinite simple groups of Richard Thompson in Sections~\ref{ss:infinite_thompson_introduction} to~\ref{ss:infinite_thompson_t}. Here we survey the ongoing work of Bleak, Donoven, Golan, Harper, Hyde and Skipper, which reveals strong parallels between the $\frac{3}{2}$-generation of these infinite simple groups and the finite simple groups. Section~\ref{ss:infinite_thompson_introduction} serves as an introduction to Thompson's groups for any reader unfamiliar with them.

This survey is based on my one-hour lecture at \emph{Groups St Andrews 2022} at the University of Newcastle, and I thank the organisers for the opportunity to present at such an enjoyable and interesting conference. I have restricted this survey to the subject of spread and have barely discussed other aspects of generation. Even regarding the spread of finite simple groups, much more could be said, especially regarding the methods involved in proving the results. Both of these omissions from this survey are discussed amply in Burness' survey article from \emph{Groups St Andrews 2017} \cite{ref:Burness19}, which is one reason for deciding to focus in this article on the progress made in the past five years.

\vspace{0.5\baselineskip}

\textbf{Acknowledgements. } The author wrote this survey when he was first a Heilbronn Research Fellow and then a Leverhulme Early Career Fellow, and he thanks the Heilbronn Institute for Mathematical Research and the Leverhulme Trust. He thanks Tim Burness, Charles Cox, Bob Guralnick, Jeremy Rickard and a referee for their helpful comments, and he also thanks Guralnick for his input on Application~3, especially his suggested proof of Theorem~\ref{thm:x_radical}.

\section{Finite Groups} \label{s:finite}

\subsection{Generating pairs} \label{ss:finite_intro}

It is easy to write down a pair of generators for each alternating group $A_n$: for instance, if $n$ is odd, then $A_n = \< (1 \, 2 \, 3), (1 \, 2 \, \dots \, n) \>$. In 1962, Steinberg \cite{ref:Steinberg62} proved that every finite simple group of Lie type is $2$-generated, by exhibiting an explicit pair of generators. In light of the Classification of Finite Simple Groups, once the sporadic groups were all shown to be $2$-generated, it became known that every finite simple group is $2$-generated \cite{ref:AschbacherGuralnick84}. Since then, numerous stronger versions of this theorem have been proved (see Burness' survey \cite{ref:Burness19}). 

Even as early as 1962, Steinberg raised the possibility of stronger versions of his $2$-generation result \cite{ref:Steinberg62}:

\trivlist\item[]\small
``It is possible that one of the generators can be chosen of order 2, as is the case for the projective unimodular group, or even that one of the generators can be chosen as an arbitrary element other than the identity, as is the case for the alternating groups. Either of these results, if true, would quite likely require methods much more detailed than those used here.''
\endtrivlist\normalsize

That is, Steinberg is suggesting the possibility that for a finite simple group $G$ one might be able to replace just the existence of $x,y \in G$ such that $\< x, y \> = G$, with the stronger statement that for all nontrivial elements $x \in G$ there exists $y \in G$ such that $\< x,y \> = G$. He alludes to the fact that this much stronger condition is known to hold for the alternating groups, which was shown by Piccard in 1939 \cite{ref:Piccard39}. In the following example, we will prove this result on alternating groups, but with different methods than Piccard used.

\begin{example} \label{ex:alternating}
Let $G = A_n$ for $n \geq 5$. We will focus on the case $n \equiv 0 \mod{4}$ and then address the remaining cases at the end. 

Write $n=4m$ and let $s$ have cycle shape $[2m-1,2m+1]$, that is, let $s$ be a product of disjoint cycles of lengths $2m-1$ and $2m+1$. Visibly, $s$ is contained in a maximal subgroup $H \leq G$ of type $(\Sm{2m-1} \times \Sm{2m+1}) \cap G$. We claim that no further maximal subgroups of $G$ contain $s$. Imprimitive maximal subgroups are ruled out since $2m-1$ and $2m+1$ are coprime. In addition, a theorem of Marggraf \cite[Theorem~13.5]{ref:Wielandt64} ensures that no proper primitive subgroup of $A_n$ contains a $k$-cycle for $k < \frac{n}{2}$, so $s$ is contained in no primitive maximal subgroups as a power of $s$ is a $(2m-1)$-cycle. 

Now let $x$ be an arbitrary nontrivial element of $G$. Choosing $g$ such that $x$ moves some point from the $(2m-1)$-cycle of $s^g$ to a point in the $(2m+1)$-cycle of $s^g$ gives $x \not\in H^g$. This means that no maximal subgroup of $G$ contains both $x$ and $s^g$, so $\< x, s^g \> = G$. In particular, every nontrivial element of $G$ is contained in a generating pair. 

We now address the other cases, but we assume that $n \geq 25$ for exposition. If $n \equiv 2 \mod{4}$, then we choose $s$ with cycle shape $[2m-1,2m+3]$ (where $n=4m+2$) and proceed as above but now the unique maximal overgroup has type $(\Sm{2m-1} \times \Sm{2m+3}) \cap G$. A similar argument works for odd $n$. Here $s$ has cycle shape $[m-2,m,m+2]$ if $n=3m$, $[m+1,m+1,m-1]$ if $n=3m+1$ and $[m+2,m,m]$ if $n=3m+2$, and the only maximal overgroups of $s$ are the three obvious intransitive ones. For each $1 \neq x \in G$, it is easy to find $g \in G$ such that $x$ misses all three maximal overgroups of $s^g$ and hence deduce that $\< x, s^g \> = G$.
\end{example}

In 2000, Guralnick and Kantor \cite{ref:GuralnickKantor00} gave a positive answer to the longstanding question of Steinberg by proving the following.

\begin{theorem} \label{thm:guralnick_kantor}
Let $G$ be a finite simple group. Then every nontrivial element of $G$ is contained in a generating pair.
\end{theorem}

We say that a group $G$ is \emph{$\frac{3}{2}$-generated} if every nontrivial element of $G$ is contained in a generating pair. The author does not know the origin of this term, but it indicates that the class of $\frac{3}{2}$-generated groups includes the class of $1$-generated groups and is included in the class of $2$-generated groups. This is somewhat analogous to the class of $\frac{3}{2}$-transitive permutation groups introduced by Wielandt \cite[Section~10]{ref:Wielandt64}, which is included in the class of $1$-transitive groups and includes the class of $2$-transitive groups.

Let us finish this section by briefly turning from simple groups to simple Lie algebras. Here we have a theorem of Ionescu \cite{ref:Ionescu76}, analogous to Theorem~\ref{thm:guralnick_kantor}.

\begin{theorem}\label{thm:ionescu}
Let $\mathfrak{g}$ be a finite dimensional simple Lie algebra over $\mathbb{C}$. Then for all $x \in \mathfrak{g} \setminus 0$ there exists $y \in \mathfrak{g}$ such that $x$ and $y$ generate $\mathfrak{g}$ as a Lie algebra.
\end{theorem}

In fact, Bois \cite{ref:Bois09} proved that every classical finite dimensional simple Lie algebra in characteristic other than 2 or 3 has this $\frac{3}{2}$-generation property, but Goldstein and Guralnick \cite{ref:GoldsteinGuralnick} have proved that $\mathfrak{sl}_n$ in characteristic 2 does not.

\subsection{Spread} \label{ss:finite_spread}

Let us now introduce the concept that gives this article its name.

\begin{definition} \label{def:spread}
Let $G$ be a group. 
\begin{enumerate}
\item The \emph{spread} of $G$, written $s(G)$, is the supremum over integers $k$ such that for any $k$ nontrivial elements $x_1, \dots, x_k \in G$ there exists $y \in G$ such that $\< x_1, y \> = \cdots = \< x_k, y \> = G$. 
\item The \emph{uniform spread} of $G$, written $u(G)$, is the supremum over integers $k$ for which there exists $s \in G$ such that for any $k$ nontrivial elements $x_1, \dots, x_k \in G$ there exists $y \in s^G$ such that $\< x_1, y \> = \cdots = \< x_k, y \> = G$.
\end{enumerate}
\end{definition}

The term spread was introduced by Brenner and Wiegold in 1975 \cite{ref:BrennerWiegold75}, but the term uniform spread was not formally introduced until 2008 \cite{ref:BreuerGuralnickKantor08}.

Note that $s(G) > 0$ if and only if every nontrivial element of $G$ is contained in a generating pair. Therefore, spread gives a way of quantifying how strongly a group is $\frac{3}{2}$-generated.  Uniform spread captures the idea that the complementary element $y$, while depending on the elements $x_1, \dots, x_k$, can be chosen somewhat uniformly for all choices of $x_1, \dots, x_k$: it can always be chosen from the same prescribed conjugacy class. In Section~\ref{ss:finite_udn}, we will see a way of measuring how much more uniformity in the choice of $y$ we can insist on. Observe that Example~\ref{ex:alternating} actually shows that $u(A_n) \geq 1$ for all $n \geq 5$.

By Theorem~\ref{thm:guralnick_kantor}, every finite simple group $G$ satisfies $s(G) > 0$. What more can be said about the (uniform) spread of finite simple groups? The main result is the following proved by Breuer, Guralnick and Kantor \cite{ref:BreuerGuralnickKantor08}.

\begin{theorem}\label{thm:breuer_guralnick_kantor}
\hspace{-2mm} Let $G$ be a nonabelian finite simple group. Then $s(G) \geq u(G) \geq 2$. Moreover, $s(G) = 2$ if and only if $u(G)=2$ if and only if
\vspace{-1.5pt}
\[
G \in \{ A_5, A_6, \Omega^+_8(2) \} \cup \{ {\rm Sp}_{2m}(2) \mid m \geq 3 \}.
\]
\end{theorem}

The asymptotic behaviour of (uniform) spread is given by the following theorem of Guralnick and Shalev \cite[Theorem~1.1]{ref:GuralnickShalev03}. The version of this theorem stated in \cite{ref:GuralnickShalev03} is given just in terms of spread, but the result given here follows immediately from their proof (see \cite[Lemma~2.1--Corollary~2.3]{ref:GuralnickShalev03}). 

\begin{theorem}\label{thm:guralnick_shalev}
Let $(G_i)$ be a sequence of nonabelian finite simple groups such that $|G_i| \to \infty$. Then $s(G_i) \to \infty$ if and only if $u(G_i) \to \infty$ if and only if $(G_i)$ has no infinite subsequence consisting of either
\begin{enumerate}
\item alternating groups of degree all divisible by a fixed prime
\item symplectic groups over a field of fixed even size or odd-dimensional orthogonal groups over a field of fixed odd size.
\end{enumerate}
\end{theorem}

Given that $s(G_i) \to \infty$ if and only if $u(G_i) \to \infty$, we ask the following. (Note that $s(G) - u(G)$ can be arbitrarily large, see Theorem~\ref{thm:spread_psl2}(iv) for example.)

\begin{question} \label{que:spread_uniform}
Does there exist a constant $c$ such that for all nonabelian finite simple groups $G$ we have $s(G) \leq c \cdot u(G)$?
\end{question}

There are explicit upper bounds that justify the exceptions in parts~(i) and~(ii) of Theorem~\ref{thm:guralnick_shalev}. Indeed, $s(\Sp_{2m}(q)) \leq q$ for even $q$ and $s(\Omega_{2m+1}(q)) \leq \frac{1}{2}(q^2+q)$ for odd $q$ (see \cite[Proposition~2.5]{ref:GuralnickShalev03} for a geometric proof). For alternating groups of composite degree $n > 4$, if $p$ is the least prime divisor of $n$, then $s(A_n) \leq \binom{2p+1}{3}$ (see \cite[Proposition~2.4]{ref:GuralnickShalev03} for a combinatorial proof). For even-degree alternating groups, the situation is clear: $s(A_n)=4$, but much less is known in odd degrees (see \cite[Section~3.1]{ref:GuralnickShalev03} for partial results). 

\begin{question} \label{que:spread_alt}
What is the (uniform) spread of $A_n$ when $n$ is odd?
\end{question}

The spread of even-degree alternating groups was determined by Brenner and Wiegold in the paper where they first introduced the notion of spread. They also studied the spread of two-dimensional linear groups, but their claimed value for $s(\PSL_2(q))$ was only proved to be a lower bound. Further work by Burness and Harper demonstrates that this is not an upper bound when $q \equiv 3 \mod{4}$, where they prove the following (see \cite[Theorem~5 \& Remark~5]{ref:BurnessHarper20}).

\begin{theorem} \label{thm:spread_psl2}
Let $G = \PSL_2(q)$ with $q \geq 11$.
\begin{enumerate}
\item If $q$ is even, then $s(G) = u(G) = q-2$.
\item If $q \equiv 1 \mod{4}$, then $s(G) = u(G) = q-1$.
\item If $q \equiv 3 \mod{4}$, then $s(G) \geq q-3$ and $u(G) \geq q-4$.
\item If $q \equiv 3 \mod{4}$ is prime, then $s(G) \geq \frac{1}{2}(3q-7)$ and $s(G)-u(G) = \frac{1}{2}(q+1)$.
\end{enumerate}
\end{theorem}

\begin{question} \label{que:spread_psl2}
What is the (uniform) spread of $\PSL_2(q)$ when $q \equiv 3 \mod{4}$?
\end{question}

In short, determining the spread of simple groups is difficult. We conclude by commenting that the precise value of the spread of only two sporadic groups is known, namely $s(\mathrm{M}_{11}) = 3$ \cite{ref:Woldar07} (see also \cite{ref:BradleyHolmes07}) and $s(\mathrm{M}_{23}) = 8064$ \cite{ref:BradleyHolmes07, ref:Fairbairn12JGT}.

In contrast, the exact spread and uniform spread of symmetric groups is known. In a series of papers in the late 1960s \cite{ref:Binder68,ref:Binder70,ref:Binder70MZ,ref:Binder73}, Binder determined the spread of $\Sm{n}$ and also showed that $u(\Sm{n}) \geq 1$ unless $n \in \{4,6\}$ (Binder used different terminology). However, the uniform spread of symmetric groups was only completely determined in a 2021 paper of Burness and Harper \cite{ref:BurnessHarper20}; indeed, showing that $u(\Sm{n}) \geq 2$ for even $n > 6$ involves both a long combinatorial argument and a CFSG-dependent group theoretic argument (see \cite[Theorem~3 \& Remark~3]{ref:BurnessHarper20}). We say more on $\Sm{6}$ in Example~\ref{ex:spread_s6}.

\begin{theorem} \label{thm:spread_sym}
Let $G = \Sm{n}$ with $n \geq 5$. Then 
\[
s(G) = \left\{\begin{array}{ll} 
2 & \text{if $n$ is even} \\
3 & \text{if $n$ is odd} \\
\end{array} \right.\quad \text{and} \quad
u(G) = \left\{\begin{array}{ll} 
0 & \text{if $n=6$} \\
2 & \text{otherwise.} \\
\end{array} \right.
\]
\end{theorem}

\vspace{0.5\baselineskip}

\textbf{Methods. A probabilistic approach. } As we turn to discuss the key method behind these results, we return to Example~\ref{ex:alternating} where we proved that $u(G) \geq 1$ when $G = A_n$ for even $n > 6$. We found an element $s \in G$ contained in a unique maximal subgroup $H$ of $G$. Since $G$ is simple, $H$ is corefree, so $\bigcap_{g \in G} H^g = 1$, which means that for each nontrivial $x \in G$ there exists $g \in G$ such that $x \not\in H^g$. This implies that $\< x, s^g \> = G$, so $s^G$ witnesses $u(G) \geq 1$. This argument can be generalised in two ways: one yields Lemma~\ref{lem:spread}, giving a better lower bound on the uniform spread of $G$ and the other yields Lemma~\ref{lem:udn_bases}, pertaining to the uniform domination number of $G$, which we will meet in the next section.

Lemma~\ref{lem:spread} takes a probabilistic approach, so we need some notation. For a finite group $G$ and elements $x,s \in G$, we write
\begin{equation} \label{eq:q}
Q(x,s) = \frac{|\{ y \in s^G \mid \< x,y \> \neq G \}|}{|s^G|},
\end{equation}
which is the probability a uniformly random conjugate of $s$ does not generate with $x$, and write $\M(G,s)$ for the set of maximal subgroups of $G$ that contain $s$.

\begin{lemma} \label{lem:spread}
Let $G$ be a finite group and let $s \in G$.
\begin{enumerate}
\item For $x \in G$, 
\[ 
Q(x,s) \leq \sum_{H \in \M(G,s)}^{} \frac{|x^G \cap H|}{|x^G|}.
\]
\item For a positive integer $k$, if $Q(x,s) < \frac{1}{k}$ for all prime order elements $x \in G$, then $u(G) \geq k$ is witnessed by $s^G$.
\end{enumerate}
\end{lemma}

\begin{proof}
For (i), let $x \in G$. Then $\<x,s^g\> \neq G$ if and only if $x \in H^g$, or equivalently $x^{g^{-1}} \in H$, for some $H \in \M(G,s)$. Therefore, \[
Q(x,s) =  \frac{|\{y \in s^G \mid \< x, y \> \neq G \}|}{|s^G|} \leq \sum_{H \in \M(G,s)} \frac{|x^G \cap H|}{|x^G|}.
\]
For (ii), fix $k$. To prove that $u(G) \geq k$ is witnessed by $s^G$, it suffices to prove that for all elements $x_1,\dots,x_k \in G$ of prime order there exists $y \in s^G$ such that $\<x_i,y\>=G$ for all $1 \leq i \leq k$. Therefore, let $x_1,\dots,x_k \in G$ have prime order. If $Q(x_i,s) < \frac{1}{k}$ for all $1 \leq i \leq k$, then 
\[
\frac{|\{y \in s^G \mid \text{$\< x_i, y \> = G$ for all $1 \leq i \leq k$} \}|}{|s^G|} \geq 1 - \sum_{i=1}^{k}Q(x_i,s) > 0,
\]
so there exists $y \in s^G$ such that $\<x_i,y\>=G$ for all $1 \leq i \leq k$.
\end{proof}

Therefore, to obtain lower bounds on the uniform spread (and hence spread) of a finite group, it is enough to (a) identify an element whose maximal overgroups $H$ are tightly constrained, and then (b) for each such $H$ and for all prime order $x \in G$, bound the quantity $\frac{|x^G \cap H|}{|x^G|}$.  

The ratio $\frac{|x^G \cap H|}{|x^G|}$ is the well-studied \emph{fixed point ratio}. More precisely, $\frac{|x^G \cap H|}{|x^G|}$ is nothing other than the proportion of points in $G/H$ fixed by $x$ in the natural action of $G$ on $G/H$. These fixed point ratios, in the context of primitive actions of almost simple groups, have seen many applications via probabilistic methods, not just to spread, but also to base sizes (e.g. the Cameron--Kantor conjecture) and monodromy groups (e.g. the Guralnick--Thompson conjecture), see Burness' survey article \cite{ref:Burness18}. 

To address task (a), one applies the well-known and extensive literature on the subgroup structure of almost simple groups. For (b), one appeals to the bounds on fixed point ratios of primitive actions of almost simple groups, the most general of which is \cite[Theorem~1]{ref:LiebeckSaxl91} of Liebeck and Saxl. This states that
\begin{equation}\label{eq:fpr}
\frac{|x^G \cap H|}{|x^G|} \leq \frac{4}{3q}
\end{equation}
for any almost simple group of Lie type $G$ over $\F_q$, maximal subgroup $H \leq G$ and nontrivial element $x \in G$, with known exceptions. This is essentially best possible, since $\frac{|x^G \cap H|}{|x^G|} \approx q^{-1}$ when $q$ is odd, $G = \mathrm{PGL}_n(q)$, $H$ is the stabiliser of a $1$-space of $\F_q^n$ and $x$ lifts to the diagonal matrix $[-1,1,1, \dots, 1] \in \GL_n(q)$. However, there are much stronger bounds that take into account the particular group $G$, subgroup $H$ or element $x$ (see \cite[Section~2]{ref:Burness18} for a survey). 

Bounding uniform spread via Lemma~\ref{lem:spread} was the approach introduced by Guralnick and Kantor in their 2000 paper \cite{ref:GuralnickKantor00} where they prove that $u(G) \geq 1$ for all nonabelian finite simple groups $G$. Clearly this approach also easily yields further probabilistic information and we refer the reader to Burness' survey article \cite{ref:Burness19} for much more on this approach. We will give just one example, which we will return to later in the article (see \cite[Example~3.9]{ref:Burness19}).

\begin{example} \label{ex:spread_e8}
Let $G = E_8(q)$ and let $s$ generate a cyclic maximal torus of order $\Phi_{30}(q) = q^8+q^7-q^5-q^4-q^3+q+1$. Weigel proved that $\M(G,s) = \{ H \}$ where $H = N_G(\<s\>) = \<s\>:30$ (see \cite[Section~4(j)]{ref:Weigel92}). Applying Lemma~\ref{lem:spread} with the bound in \eqref{eq:fpr}, for all nontrivial $x \in G$ we have $u(G) \geq 1$ since
\[
\sum_{H \in \M(G,s)} \frac{|x^G \cap H|}{|x^G|} \leq \frac{4}{3q} \leq \frac{2}{3} < 1.
\]
However, we can do better: $|x^G \cap H| \leq |H| \leq q^{14}$ and $|x^G| > q^{58}$ for all nontrivial elements $x \in G$, so $u(G) \geq q^{44}$ since
\[
\sum_{H \in \M(G,s)} \frac{|x^G \cap H|}{|x^G|} < \frac{1}{q^{44}}.
\]
\end{example}

While the overwhelming majority of results on (uniform) spread are established via the probabilistic method encapsulated in Lemma~\ref{lem:spread}, there are cases where this approach fails, as the following example highlights.

\begin{example} \label{ex:spread_sp}
Let $m \geq 3$ and let $G = \Sp_{2m}(2)$. By Theorem~\ref{thm:breuer_guralnick_kantor}, we know that $u(G)=2$. However, if $x$ is a transvection, then $Q(x,s) > \frac{1}{2}$ for all $s \in G$. This is proved in \cite[Proposition~5.4]{ref:BreuerGuralnickKantor08}, and we give an indication of the proof. Every element of $G = \Sp_{2m}(2)$ is contained in a subgroup of type ${\rm O}^+_{2m}(2)$ or ${\rm O}^-_{2m}(2)$ (see \cite{ref:Dye79}, for example). 

Assume that $s$ is contained in a subgroup $H \cong {\rm O}^-_{2m}(2)$. The groups $\Sp_{2m}(2)$ and ${\rm O}^\pm_{2m}(2)$ contain $2^{2m}-1$ and $2^{2m-1} \mp 2^{m-1}$ transvections, respectively, so
\[
Q(x,s) \geq \frac{2^{2m-1}+2^{m-1}}{2^{2m}-1} = \frac{2^{m-1}}{2^m-1} > \frac{1}{2}.
\]

A more involved argument gives $Q(x,s) > \frac{1}{2}$ if $s$ is contained in a subgroup of type ${\rm O}^+_{2m}(2)$ but none of type ${\rm O}^-_{2m}(2)$, relying on $s$ being reducible here.
\end{example}

\subsection{Uniform domination} \label{ss:finite_udn}

We began by observing that any finite simple group $G$ is $\frac{3}{2}$-generated, that is
\begin{equation} \label{eq:generation}
\text{for all $x \in G \setminus 1$ there exists $y \in G$ such that $\<x,y\> = G$.}
\end{equation}
We then looked to strengthen \eqref{eq:generation} by increasing the scope of the first quantifier. Recall that the \emph{spread} of $G$, denoted $s(G)$, is the greatest $k$ such that
\[
\text{for all $x_1, \dots, x_k \in G$ there exists $y \in G$ such that $\<x_1,y\> = \cdots = \< x_k,y\> = G$.} 
\]
We also have a related notion: the \emph{uniform spread} of $G$, denoted $u(G)$, is the greatest $k$ for which there exists an element $s \in G$ such that
\[
\text{for all $x_1, \dots, x_k \in G$ there exists $y \in s^G$ such that $\<x_1,y\> = \cdots = \< x_k,y\> = G$.}
\]
The notion of uniform spread inspires us to strengthen \eqref{eq:generation} by narrowing the range of the second quantifier. That is, we say that the \emph{total domination number} of $G$, denoted $\gamma_t(G)$, is the least size of a subset $S \subseteq G$ such that
\[
\text{for all $x \in G \setminus 1$ there exists $y \in S$ such that $\<x,y\> = G$.}
\]
Again we have a related notion: the \emph{uniform domination number} of $G$, denoted $\gamma_u(G)$, is the least size of a subset $S \subseteq G$ of conjugate elements such that
\[
\text{for all $x \in G \setminus 1$ there exists $y \in S$ such that $\<x,y\> = G$.}
\]
These latter two concepts were introduced by Burness and Harper in \cite{ref:BurnessHarper19} and studied further in \cite{ref:BurnessHarper20}. The terminology is motivated by the generating graph (see Section~\ref{ss:finite_graph}).

Let $G$ be a nonabelian finite simple group. Clearly $2 \leq \gamma_t(G) \leq \gamma_u(G)$, and since $u(G) \geq 1$, there exists a conjugacy class $s^G$ such that $\gamma_u(G) \leq |s^G|$. However, the class exhibited in Guralnick and Kantor's proof of $u(G) \geq 1$ is typically very large (for groups of Lie type, $s$ is usually a regular semisimple element), so it is natural to seek tighter upper bounds on $\gamma_u(G)$. The following result of Burness and Harper does this \cite[Theorems~2, 3 \& 4]{ref:BurnessHarper19} (see \cite[Theorem~4(i)]{ref:BurnessHarper20} for the refined upper bound in (iii)).

\newpage
\begin{theorem} \label{thm:udn}
Let $G$ be a nonabelian finite simple group. 
\begin{enumerate}
\item If $G = A_n$, then $\gamma_u(G) \leq 77 \log_2{n}$.
\item If $G$ is classical of rank $r$, then $\gamma_u(G) \leq 7r+70$.
\item If $G$ is exceptional, then $\gamma_u(G) \leq 5$.
\item If $G$ is sporadic, then $\gamma_u(G) \leq 4$.
\end{enumerate}
\end{theorem}

In this generality, these bounds are optimal up to constants. For example, if $n \geq 6$ is even, then $\log_2{n} \leq \gamma_t(A_n) \leq \gamma_u(A_n) \leq 2 \log_2{n}$, and if $G$ is $\Sp_{2r}(q)$ with $q$ even or $\Omega_{2r+1}(q)$ with $q$ odd, then $r \leq \gamma_t(G) \leq \gamma_u(G) \leq 7r$ \cite[Theorems~3(i) \& 6.3(iii)]{ref:BurnessHarper20}. Regarding the bounds in (iii) and (iv), for sporadic groups, $\gamma_u(G) = 4$ is witnessed by $G = {\rm M}_{11}$  \cite[Theorem~3]{ref:BurnessHarper19}, but the best lower bound for exceptional groups is $\gamma_u(G) \geq 3$ given by $G = F_4(q)$ \cite[Lemma~6.17]{ref:BurnessHarper20}.

\begin{question} \label{que:udn_alt}
Does there exist a constant $c$ such that for all $n \geq 5$ we have $\log_p{n} \leq \gamma_t(A_n) \leq \gamma_u(A_n) \leq c \log_p{n}$ where $p$ is the least prime divisor of $n$?
\end{question}

By \cite[Theorem~4(ii)]{ref:BurnessHarper20}, we know that $\gamma_t(A_n) \geq \log_p{n}$, so to provide an affirmative answer to Question~\ref{que:udn_alt}, it suffices to prove that $\gamma_u(A_n) \leq c \log_p{n}$.

\begin{question} \label{que:udn_lie}
Does there exist a constant $c$ such that for all finite simple groups of Lie type $G$ other than $\Sp_{2r}(q)$ with $q$ even and $\Omega_{2r+1}(q)$ with $q$ odd, we have $\gamma_u(G) \leq c$?
\end{question}

By Theorem~\ref{thm:udn}, to answer Question~\ref{que:udn_lie}, it suffices to consider classical groups of large rank, and it was shown in \cite[Theorem~6.3(ii)]{ref:BurnessHarper19} that $c=15$ suffices for some families of these groups. Affirmative answers to Questions~\ref{que:udn_alt} and~\ref{que:udn_lie} would answer Question~\ref{que:udn_tdn} too.

\begin{question} \label{que:udn_tdn}
Does there exist a constant $c$ such that for all nonabelian finite simple groups $G$ we have $\gamma_u(G) \leq c \cdot \gamma_t(G)$?
\end{question}

The smallest possible value of $\gamma_u(G)$ is $2$ (since $G$ is not cyclic), and an almost complete classification of when this is achieved was given in \cite[Corollary~7]{ref:BurnessHarper20}.

\begin{theorem} \label{thm:udn_two}
Let $G$ be a nonabelian finite simple group. Then $\gamma_u(G) = 2$ only if $G$ is one of the following
\begin{enumerate}
\item $A_n$ for prime $n \geq 13$
\item $\PSL_2(q)$ for odd $q \geq 11$
\item[] $\PSL^\e_n(q)$ for odd $n$, but not $n=3$ with $(q,\e) \in \{ (2,+), (4,+), (3,-), (5,-) \}$
\item[] ${\rm PSp}_{4m+2}(q)^\ast$ for odd $q$ and $m \geq 2$, and ${\rm P}\Omega^\pm_{4m}(q)^\ast$ for $m \geq 2$ 
\item ${}^2B_2(q)$,  ${}^2G_2(q)$, ${}^2F_4(q)$, ${}^3D_4(q)$, ${}^2E_6(q)$, $E_6(q)$, $E_7(q)$ , $E_8(q)$ 
\item ${\rm M}_{23}$, ${\rm J}_1$, ${\rm J}_4$, ${\rm Ru}$, ${\rm Ly}$, ${\rm O'N}$, ${\rm Fi}_{23}$, ${\rm Th}$, $\mathbb{B}$, $\mathbb{M}$ or ${\rm J}_3^\ast$, ${\rm He}^\ast$, ${\rm Co}_1^\ast$, ${\rm HN}^\ast$.
\end{enumerate}
Moreover, $\gamma_u(G) = 2$ in all the cases without an asterisk.
\end{theorem}

We will say that a subset $S \subseteq G$ of conjugate elements of $G$ is a \emph{uniform dominating set} of $G$ if for all nontrivial $x \in G$ there exists $y \in S$ such that $\< x, y \> = G$, so $\gamma_u(G)$ is the smallest size of a uniform dominating set of $G$. For groups $G$ such that $\gamma_u(G) = 2$, we know that there exists a uniform dominating set of size two. How abundant are such subsets? To this end, let $P(G,s,2)$ be the probability that two random conjugates of $s$ form a uniform dominating set for $G$, and let $P(G) = \max\{ P(G,s,2) \mid s \in G \}$. Then we have the following probabilistic result \cite[Corollary~8 \& Theorem~9]{ref:BurnessHarper20}.

\begin{theorem} \label{thm:udn_prob}
Let $(G_i)$ be a sequence of nonabelian finite simple groups such that $|G_i| \to \infty$. Assume that $\gamma_u(G_i) = 2$, and $G_i \not\in \{ {\rm PSp}_{4m+2}(q) \mid \text{odd $q$, $m \geq 2$}\} \cup \{ {\rm P}\Omega^\pm_{4m}(q) \mid \text{$m \geq 2$} \} \cup \{ {\rm J}_3, {\rm He}, {\rm Co}_1, {\rm HN} \}$. Then  \vspace{-3pt}
\[
P(G_i) \to \left\{ \begin{array}{ll} \frac{1}{2} & \text{if $G = \PSL_2(q)$} \\ 1 & \text{otherwise.} \end{array} \right. \vspace{-3pt}
\]
Moreover, $P(G_i) \leq \frac{1}{2}$ if and only if $G_i =\PSL_2(q)$ for $q \equiv 3 \mod{4}$ or $G_i \in \{ A_{13}, {\rm PSU}_5(2), {\rm Fi}_{23} \}$.
\end{theorem}

\vspace{0.5\baselineskip}

\textbf{Methods. Bases of permutation groups. } 
Let us now discuss the methods used in \cite{ref:BurnessHarper19,ref:BurnessHarper20} to bound $\gamma_u(G)$. Here there is a very pleasing connection with an entirely different topic in permutation group theory: bases. For a group $G$ acting faithfully on a set $\Omega$, a subset $B \subseteq \Omega$ is a \emph{base} if the pointwise stabiliser $G_{(B)}$ is trivial. Since $G$ acts faithfully, the entire domain $\Omega$ is a base, so we naturally ask for the smallest size of a base, which we call the \emph{base size} $b(G,\Omega)$. To turn this combinatorial notion into an algebraic one, we observe that when $G$ acts on $G/H$, a subset $\{ Hg_1, \dots, Hg_c \}$ is a base if and only if $\cap_{i=1}^c H^{g_i} = 1$, so $b(G,G/H)$ is the smallest number of  conjugates of $H$ whose intersection is trivial. 

Bases have been studied for over a century, and the base size has been at the centre of several recently proved conjectures, such as Pyber's conjecture that there is a constant $c$ such that $\frac{\log|G|}{\log|\Omega|} \leq b(G, \Omega) \leq c \frac{\log|G|}{\log|\Omega|}$ for all primitive groups $G \leq \Sym(\Omega)$ (see \cite{ref:DuyanHalasiMaroti18}), and Cameron's conjecture that $b(G, \Omega) \leq 7$ for nonstandard primitive almost simple groups $G \leq \Sym(\Omega)$ (see \cite{ref:BurnessLiebeckShalev09}). There is an ambitious ongoing programme of work, initiated by Saxl, to provide a complete classification of the primitive groups $G \leq \Sym(\Omega)$ with $b(G,\Omega) = 2$. There are numerous partial results in this direction, and we give just one, as we will use it below. Burness and Thomas \cite{ref:BurnessThomas} proved that if $G$ is a simple group of Lie type and $T$ is a maximal torus, then $b(G,G/N_G(T)) = 2$ apart from a few known low rank exceptions.

The following result is the bridge that connects bases with uniform domination (see \cite[Corollaries~2.2 \&~2.3]{ref:BurnessHarper19}).

\begin{lemma} \label{lem:udn_bases}
Let $G$ be a finite group and let $s \in G$. 
\begin{enumerate}
\item Assume that $\M(G,s) = \{ H \}$ and $H$ is corefree. Then the smallest uniform dominating set $S \subseteq s^G$ satisfies $|S|=b(G,G/H)$.
\item Assume that $H \in \M(G,s)$ is corefree. Then every uniform dominating set $S \subseteq s^G$ satisfies $|S| \geq b(G,G/H)$.
\end{enumerate}
\end{lemma}

\begin{proof}
For (i), note that $x \in H$ if and only if $\< x, s \> \neq G$. Hence, $\{ s^{g_1}, \dots, s^{g_c} \}$ is a uniform dominating set if and only if $\bigcap_{i=1}^{c} H^{g_i} = 1$, or said otherwise, if and only if $\{ g_1, \dots, g_c \}$ is a base for $G$ acting on $G/H$. The result follows.

For (ii), if $x \in H$, then $\< x, s \> \neq G$. Therefore, if $\{ s^{g_1}, \dots, s^{g_c} \}$ is a uniform dominating set, then $\bigcap_{i=1}^{c} H^{g_i} = 1$, so $\{ g_1, \dots, g_c \}$ is a base for $G$ acting on $G/H$ and, consequently, $c \geq b(G, G/H)$.
\end{proof}

Let us explain how Lemma~\ref{lem:udn_bases} applies. Part~(i) gives an upper bound: if we can find $s \in G$ such that $\M(G,s) = \{H\}$ and $b(G,G/H) \leq c$, then $\gamma_u(G) \leq c$. Part~(ii) gives a lower bound: if we can show that for all $s \in G$ there exists $H \in \M(G,s)$ with $b(G,G/H) \geq c$, then $\gamma_u(G) \geq c$. We give two examples to show how we do this in practice.

\begin{example} \label{ex:udn_e8}
Let $G = E_8(q)$ and let $s$ generate a cyclic maximal torus of order $\Phi_{30}(q) = q^8+q^7-q^5-q^4-q^3+q+1$. As noted in Example~\ref{ex:spread_e8}, $\M(G,s) = \{ H \}$ where $H$ is the normaliser of the torus $\<s\>$. Now, applying Burness and Thomas' result \cite[Theorem~1]{ref:BurnessThomas} mentioned above, we see that $b(G,G/H)=2$, so Lemma~\ref{lem:udn_bases} implies that $\gamma_u(G) = 2$.
\end{example}

\begin{example} \label{ex:udn_an}
Let $n > 6$ be even and let $G = A_n$. We will give upper and lower bounds on $\gamma_u(G)$ via Lemma~\ref{lem:udn_bases}.

Seeking an upper bound on $\gamma_u(G)$, let $s = (1 \, 2 \, \dots \, l)(l+1 \, l+2 \, \dots \, n)$ where $l \in \{\frac{n}{2}-1, \frac{n}{2}-2\}$ is odd. As we showed in Example~\ref{ex:alternating}, $\M(G,s) = \{H\}$ where $H \cong (\Sm{l} \times \Sm{n-l}) \cap A_n$. The action of $A_n$ on $A_n/H$ is just the action of $A_n$ on the set of $l$-subsets of $\{1, 2, \dots, n\}$. The base size of this action was studied by Halasi, and by \cite[Theorem~4.2]{ref:Halasi12}, we have $b(G,G/H) \leq \left\lceil \log_{\lceil n/l \rceil} n \right\rceil \cdot (\lceil n/l \rceil - 1) \leq 2 \log_2{n}$. Applying Lemma~\ref{lem:udn_bases}(i) gives $\gamma_u(G) \leq 2\log_2{n}$.

Turning to a lower bound, note that every element of $G$ is contained in a subgroup $K$ of type $(\Sm{k} \times \Sm{n-k}) \cap A_n$ for some $0 < k < n$. By \cite[Theorem~3.1]{ref:Halasi12}, we have $b(G, G/K) \geq \log_2{n}$. Applying Lemma~\ref{lem:udn_bases}(ii) gives $\gamma_u(G) \geq \log_2{n}$.
\end{example}

We now address the general case where $s$ is not contained in a unique maximal subgroup of $G$. In the spirit of how uniform spread was studied, a probabilistic approach is adopted. Write $Q(G,s,c)$ for the probability that a random $c$-tuple of elements of $s^G$ does not give a uniform dominating set of $G$ and write $\mathcal{P}(G)$ for the set of prime order elements of $G$. The main lemma is \cite[Lemma~2.5]{ref:BurnessHarper19}.

\begin{lemma} \label{lem:udn_prob}
Let $G$ be a finite group, let $s \in G$ and let $c$ be a positive integer.
\begin{enumerate}
\item For all positive integers $c$, we have
\[
Q(G,s,c) \leq \sum_{x \in \mathcal{P}(G)} \left(\sum_{H \in \M(G,s)}\frac{|x^G \cap H|}{|x^G|}\right)^c.
\]
\item For a positive integer $c$, if $Q(G,s,c) < 1$, then $\gamma_u(G) \leq c$.
\end{enumerate}
\end{lemma}

\begin{proof}
Part~(ii) is immediate. For part~(i), $\{s^{g_1}, \dots, s^{g_c}\}$ is not a uniform dominating set of $G$ if and only if there exists a prime order element $x \in G$ such that $\< x, s^{g_i} \> \neq G$ for all $1 \leq i \leq c$. Since $Q(x,s)$ is the probability that $x$ does not generate $G$ with a random conjugate of $s$ (see \eqref{eq:q}), this implies that $Q(G,s,c)  \leq \sum_{x \in \mathcal{P}(G)}^{} Q(x,s)^c$. The result follows from Lemma~\ref{lem:spread}(i).
\end{proof}

Could Lemma~\ref{lem:udn_prob} yield a better bound than Lemma~\ref{lem:udn_bases} when $s$ satisfies $\M(G,s) = \{H\}$? In this case, $Q(G,s,c)$ is nothing other than the probability that a random $c$-tuple of elements of $G/H$ form a base and $\sum_{x \in \mathcal{P}(G)} \left(\frac{|x^G \cap H|}{|x^G|}\right)^c$ is the upper bound for $Q(G,s,c)$ used, first by Liebeck and Shalev \cite{ref:LiebeckShalev99} and then by numerous others since, to obtain upper bounds on the base size $b(G,G/H)$. Therefore, Lemma~\ref{lem:udn_prob} has nothing new to offer in this special case.

We conclude with an example of Lemma~\ref{lem:udn_prob} in action. This establishes a (typical) special case of Theorem~\ref{thm:udn}(ii).

\begin{example} \label{ex:udn_prob}
Let $n \geq 10$ be even and let $G = \PSL_n(q)$. We proceed similarly to Example~\ref{ex:udn_an}, by fixing odd $l \in \{\frac{n}{2}-1, \frac{n}{2}-2\}$ and then letting $s$ lift to a block diagonal matrix \scalebox{0.75}{$\left( \begin{array}{cc} A & 0 \\ 0 & B \end{array} \right)$} where $A \in \mathrm{SL}_l(q)$ and $B \in \mathrm{SL}_{n-l}(q)$ are irreducible. The order of $s$ is divisible by a \emph{primitive prime divisor} of $q^{n-l}-1$ (a prime divisor coprime to $q^k-1$ for each $1 \leq k < n-l$). Using the framework of Aschbacher's theorem on the subgroup structure of classical groups \cite{ref:Aschbacher84}, Guralnick, Penttila, Praeger and Saxl, classify the subgroups of $\GL_n(q)$ that contain an element whose order is a primitive prime divisor of $q^m-1$ when $m > \frac{n}{2}$ \cite{ref:GuralnickPenttilaPraegerSaxl97}. With this we deduce that $\M(G,s) = \{ G_U, G_V \}$ where $U$ and $V$ are the obviously stabilised subspaces of dimension $l$ and $n-l$, respectively. The fixed point ratio for classical groups acting on the set of $k$-subspaces of their natural module was studied by Guralnick and Kantor, and by \cite[Proposition~3.1]{ref:GuralnickKantor00}, if $G = \PSL_n(q)$ and $H$ is the stabiliser of a $k$-subspace of $\F_q^n$, then $\frac{|x^G \cap H|}{|x^G|} < \frac{2}{q^k}$. Applying Lemma~\ref{lem:udn_prob} gives $\gamma_u(G) \leq 2n+15$ since
\[
Q(G,s,2n+15) \leq \sum_{x \in \mathcal{P}(G)} \left(\sum_{H \in \M(G,s)}\frac{|x^G \cap H|}{|x^G|}\right)^{2n+15} \!\!\leq q^{n^2-1} \left( 2 \cdot \frac{2}{q^{\frac{n}{2}-2}} \right)^{2n+15} \!\!< 1.
\]
\end{example}

\subsection{The spread of a finite group} \label{ss:finite_bgh}

We now look beyond finite simple groups and ask the general question: for which finite groups $G$ is every nontrivial element contained in a generating pair? Brenner and Wiegold's original 1975 paper gives a comprehensive answer for finite soluble groups (see \cite[Theorem~2.01]{ref:BrennerWiegold75} for even more detail).

\begin{theorem} \label{thm:brenner_wiegold}
Let $G$ be a finite soluble group. The following are equivalent:
\begin{enumerate}
\item $s(G) \geq 1$
\item $s(G) \geq 2$
\item every proper quotient of $G$ is cyclic.
\end{enumerate}
\end{theorem}

The equivalence of (i) and (ii) shows that $s(G)=1$ for no finite soluble groups $G$, so Brenner and Wiegold asked the following question \cite[Problem~1.04]{ref:BrennerWiegold75}.

\begin{question} \label{que:brenner_wiegold}
Which finite groups $G$ satisfy $s(G) = 1$? In particular, are there perhaps only finitely many such groups?
\end{question}

The condition in (iii) is necessary for every nontrivial element of $G$ to be contained in a generating pair, and this is true for an arbitrary group $G$. To see this, assume that every nontrivial element of $G$ is contained in a generating pair. Let $1 \neq N \leqn G$ and let $1 \neq n \in N$. Then there exists $g \in G$ such that $G = \< n, g \>$, so $G/N = \<Nn, Ng\> = \<Ng\>$, which is cyclic. In 2008, Breuer, Guralnick and Kantor conjectured that this condition is also sufficient for all finite groups \cite[Conjecture~1.8]{ref:BreuerGuralnickKantor08}.

\begin{conjecture} \label{con:breuer_guralnick_kantor}
Let $G$ be a finite group. Then $s(G) \geq 1$ if and only if every proper quotient of $G$ is cyclic.
\end{conjecture}

Completing a long line of research in this direction, both Question~\ref{que:brenner_wiegold} and Conjecture~\ref{con:breuer_guralnick_kantor} were settled by Burness, Guralnick and Harper in 2021 \cite{ref:BurnessGuralnickHarper21}.

\begin{theorem} \label{thm:burness_guralnick_harper}
Let $G$ be a finite group. Then $s(G) \geq 2$ if and only if every proper quotient of $G$ is cyclic.
\end{theorem}

\begin{corollary} \label{cor:burness_guralnick_harper}
No finite group $G$ satisfies $s(G)=1$.
\end{corollary}

The next example (which is \cite[Remark~2.16]{ref:BurnessGuralnickHarper21}) shows that Theorem~\ref{thm:burness_guralnick_harper} does not hold if spread is replaced with uniform spread (recall Theorem~\ref{thm:spread_sym}).

\begin{example} \label{ex:spread_s6}
Let $G = \Sm{n}$ where $n \geq 6$ is even. Suppose that $u(G) > 0$ is witnessed by the class $s^G$. Since a conjugate of $s$ generates with $(1 \, 2 \, 3)$, $s$ must be an odd permutation. Since a conjugate of $s$ generates with $(1 \, 2)$, $s$ must have at most two cycles. Since $n$ is even, it follows that $s$ is an $n$-cycle. However, if $n=6$ and $a \in \Aut(G) \setminus G$, then $s^a \in (1 \, 2 \, 3)(4 \, 5)^G$ also witnesses $u(G) > 0$, which is a contradiction. Therefore, $u(\Sm{6})=0$.
\end{example}

However, Example~\ref{ex:spread_s6} is essentially the only obstacle to a result for uniform spread analogous to Theorem~\ref{thm:burness_guralnick_harper} on spread. Indeed, Burness, Guralnick and Harper gave the following complete description of the finite groups $G$ with $u(G) < 2$ \cite[Theorem~3]{ref:BurnessGuralnickHarper21}. (Note that the uniform spread of abelian groups $G$ is not interesting: $u(G) = \infty$ if $G$ is cyclic and $u(G) = 0$ otherwise.)

\begin{theorem} \label{thm:burness_guralnick_harper_uniform}
Let $G$ be a nonabelian finite group such that every proper quotient is cyclic. Then
\begin{enumerate}
\item $u(G) = 0$ if and only if $G = \Sm{6}$
\item $u(G) = 1$ if and only if the group $G$ has a unique minimal normal subgroup $N = T_1 \times \cdots \times T_k$ where $k \geq 2$ and where $T_i = A_6$ and $N_G(T_i)/C_G(T_i) = \Sm{6}$ for all $1 \leq i \leq k$.
\end{enumerate}
\end{theorem}

Theorem~\ref{thm:burness_guralnick_harper_uniform} emphasises the anomalous behaviour of $\Sm{6}$: it is the only almost simple group $G$ where every proper quotient of $G$ is cyclic but $u(G) < 2$.

\vspace{0.5\baselineskip}
 
\textbf{Methods. A reduction theorem and Shintani descent. } We now outline the proof of Theorems~\ref{thm:burness_guralnick_harper} and~\ref{thm:burness_guralnick_harper_uniform} in \cite{ref:BurnessGuralnickHarper21}. We need to consider the finite groups $G$ all of whose proper quotients are cyclic. In light of Theorem~\ref{thm:brenner_wiegold}, we will assume that $G$ is insoluble, so $G$ has a unique minimal normal subgroup $T^k$ for a nonabelian simple group $T$, and we can assume that $G = \< T^k, s \>$ where $s = (a, 1, \dots, 1)\sigma \in \Aut(T^k)$ for $a \in \Aut(T)$ and $\sigma = (1 \, 2 \, \dots \, k) \in \Sm{k}$. Let us ignore $T = A_6$ due to the complications we have already seen in this case.

The first major step in the proof of Theorems~\ref{thm:burness_guralnick_harper} and~\ref{thm:burness_guralnick_harper_uniform} is the following reduction theorem \cite[Theorem~2.13]{ref:BurnessGuralnickHarper21}.

\begin{theorem} \label{thm:bgh_reduction}
Fix a nonabelian finite simple group $T$ and assume that $T \neq A_6$. Fix $s = (a, 1, \dots, 1)\sigma \in \Aut(T^k)$ with $a \in \Aut(T)$ and $\sigma = (1 \, 2 \, \dots \, k) \in \Sm{k}$. Then $s$ witnesses $u(\<T^k,s\>) \geq 2$ if the following hold:
\begin{enumerate}
\item $a$ witnesses $u(\< T, a \>) \geq 2$
\item $\< a \> \cap T \neq 1$, and if $a$ is square in $\Aut(T)$, then $|\<a\> \cap T|$ does not divide $4$.
\end{enumerate}
\end{theorem}

The second major step is to generalise Breuer, Guralnick and Kantor's result Theorem~\ref{thm:breuer_guralnick_kantor} that $u(T) \geq 2$ for all nonabelian finite simple groups $T$ to all almost simple groups $A = \< T, a \>$. This long line of research was initiated by Burness and Guest for $T = \PSL_n(q)$ \cite{ref:BurnessGuest13}, continued by Harper for the remaining classical groups $T$ \cite{ref:Harper17,ref:HarperLNM} and completed by Burness, Guralnick and Harper for exceptional groups $T$ \cite{ref:BurnessGuralnickHarper21}. (We have already noted that the proof of $u(S\!_n) \geq 2$ for $n \neq 6$ was completed by Burness and Harper in \cite{ref:BurnessHarper20}, and the result for sporadic groups follows from computational work in \cite{ref:BreuerGuralnickKantor08}).

We conclude by highlighting the major obstacle that this body of work faced and then outlining the technique that overcame this obstacle. 

Let $G = \< T, g \>$ where $T$ is a finite simple group of Lie type and $g \in \Aut(T)$. Suppose $s^G$ witnesses $u(G) \geq 2$. A conjugate of $s$ generates with any element of $T$, so $G = \< T, s \>$. By replacing $s$ with a power if necessary, we may assume that $s \in Tg$. How do we describe elements of $Tg$ and their overgroups? Suppose that $T = \PSL_n(q)$ with $q=p^f$. If $g \in {\rm PGL}_n(q)$, then there are geometric techniques available, but what if $g$ is, say, the field automorphism $(a_{ij}) \mapsto (a_{ij}^p)$?

The technique that was used to answer these questions is known as \emph{Shintani descent}. This was introduced by Shintani in 1976 \cite{ref:Shintani76} (and generalised by Kawanaka in \cite{ref:Kawanaka77}) to study irreducible characters of almost simple groups. However, as first exploited by Fulman and Guralnick in their work on the Boston--Shalev conjecture \cite{ref:FulmanGuralnick12}, Shintani descent also provides a fruitful way of studying the conjugacy classes of almost simple groups. The main theorem is the following, and we follow Desphande's proof \cite{ref:Deshpande16}. (Here $\sigma_i$ is considered as an element of $\< X, \sigma_i \>$ for $i \in \{1,2\}$.)

\begin{theorem} \label{thm:shintani}
Let $X$ be a connected algebraic group, and let $\sigma_1, \sigma_2\: X \to X$ be commuting Steinberg endomorphisms. Then there is a bijection 
\[
F\: \{ \text{$X_{\sigma_1}$-classes in $X_{\sigma_1}\sigma_2$} \} \to \{ \text{$X_{\sigma_2}$-classes in $X_{\sigma_2}\sigma_1$} \}.
\]
\end{theorem}

\begin{proof}
Let $S$ be the orbits of $\{ (g,h) \in X\sigma_2 \times X\sigma_1 \mid [g,h]=1 \}$ under the conjugation action of $X$. By the Lang--Steinberg theorem, $S$ is in bijection with the orbits of $\{ (x\sigma_2, \sigma_1) \mid x \in X_{\sigma_1} \}$ under conjugation by $X_{\sigma_1}$ and also with the orbits of $\{ (\sigma_2, y\sigma_1) \mid x \in X_{\sigma_2} \}$ under conjugation by $X_{\sigma_2}$. This provides a bijection between the $X_{\sigma_1}$-classes in $X_{\sigma_1}\sigma_2$ and the $X_{\sigma_2}$-classes in $X_{\sigma_2}\sigma_1$.
\end{proof}

The bijection in the proof of Theorem~\ref{thm:shintani}, known as the \emph{Shintani map} of $(X,\sigma_1,\sigma_2)$, has desirable properties. For instance, if $\sigma_1 = \sigma_2^e$ for $e \geq 1$, then
\begin{equation} \label{eq:shintani}
F((x\sigma_2)^{X_{\sigma_1}}) = (a^{-1}(x\sigma_1)^{-e}a)^{X_{\sigma_2}}
\end{equation} 
for some $a \in X$. Moreover, if $F(g^{X_{\sigma_1}}) = h^{X_{\sigma_2}}$, then it is easy to show that $C_{X_{\sigma_1}}(g) \cong C_{X_{\sigma_2}}(h)$, and, by now, extensive information is also available about how maximal overgroups of $g$ in $\< X_{\sigma_1}, \sigma_2 \>$ relate to the maximal overgroups of $h$ in $\< X_{\sigma_2}, \sigma_1 \>$. These latter results are crucial to proving Theorem~\ref{thm:burness_guralnick_harper} for almost simple groups of Lie type. The first result in this direction is due to Burness and Guest \cite[Corollary~2.15]{ref:BurnessGuest13}, and the subsequent developments are unified by Harper in \cite{ref:Harper21}, to which we refer the reader for further detail.

\begin{example} \label{ex:bgh}
We sketch how $u(G) \geq 2$ was proved for $G = \< T, g \>$ when $T = \Omega^+_{2m}(q)$ with $q=2^f$ and $g$ is the field automorphism $\p\:(a_{ij}) \mapsto (a_{ij}^2)$. We will not make further assumptions on $q$, but we will assume that $m$ is large. 

Let $X$ be the simple algebraic group ${\rm SO}_{2m}(\overline{\F}_2)$ and let $F$ be the Shintani map 

\noindent of $(X,\p^f,\p)$. Then $G = \< X_{\p^f}, \p \>$, and writing $G_0 = X_{\p} = \Omega^+_{2m}(2)$, we observe that $F$ gives a bijection between the conjugacy classes in $Tg$ and those in $G_0$. 

We define $s \in Tg$ such that $F(s^G) = s_0^{G_0}$ for a well chosen element $s_0 \in G_0$. In particular, \eqref{eq:shintani} implies that $s_0$ is $X$-conjugate to a power of $s$. To define $s_0$, fix $k$ such that $m-k$ is even and $\frac{\sqrt{2m}}{4} < 2k < \frac{\sqrt{2m}}{2}$, fix $A \in \Omega^-_{2k}(2)$ and $B \in \Omega^-_{2m-2k}(2)$ of order $2^k+1$ and $2^{m-k}+1$ and let $s_0 = \text{\scalebox{0.75}{$\left( \begin{array}{cc} A & 0 \\ 0 & B \end{array} \right)$}} \in \Omega^+_{2m}(2)$.

We now study $\M(G,s)$. First, a power of $s_0$ (and hence $s$) has a $1$-eigenspace of codimension $2k < \frac{\sqrt{2m}}{2}$, so \cite[Theorem~7.1]{ref:GuralnickSaxl03} implies that $s$ is not contained in any local or almost simple maximal subgroup of $G$. 

Next, a power of $s_0$ has order $2^{m-k}+1$, which is divisible by the \emph{primitive part} of $2^{2m-2k}-1$ (the largest divisor of $2^{2m-2k}-1$ that is prime to $2^l-1$ for all $0 < l < 2m-2k$). For sufficiently large $m$, we can apply the main theorem of \cite{ref:GuralnickPenttilaPraegerSaxl97} to deduce that all of the maximal overgroups of $s_0$ in $G_0 = \Omega^+_{2m}(2)$ are reducible. In particular, the only maximal overgroup of $s_0$ in $G_0$ arising as the set of fixed points of a closed positive-dimensional $\p$-stable subgroup of $X$ is the obvious reducible subgroup of type $({\rm O}^-_{2k}(2) \times {\rm O}^-_{2m-2k}(2)) \cap G_0$. Now the theory of Shintani descent \cite[Theorem~4]{ref:Harper21} implies that the only such maximal overgroup of $s$ in $G$ is one subgroup $H$ of type $({\rm O}^\pm_{2k}(q) \times {\rm O}^\pm_{2m-2k}(q)) \cap G$. 

Drawing these observations together and using Aschbacher's subgroup structure theorem \cite{ref:Aschbacher84}, we deduce that $\M(G,s) = \{H\} \cup \M'$, where $\M'$ consists of subfield subgroups. There are at most $\log_2{f}+1 = \log_2\log_2 q + 1$ classes of maximal subfield subgroups of $G$, and by \cite[Lemma~2.19]{ref:BurnessGuest13}, $s$ is contained in at most $|C_{G_0}(s_0)| = (2^k+1)(2^{2m-2k}+1)$ conjugates of a fixed maximal subgroup.

Using the fixed point ratio bound proved for reducible subgroups in \cite{ref:GuralnickKantor00} and irreducible subgroups in \cite{ref:Burness071,ref:Burness072,ref:Burness073,ref:Burness074}, for all nontrivial $x \in G$ we have
\[
Q(x,s) < \sum_{H \in \M(G,s)} \frac{|x^G \cap H|}{|x^G|} < \frac{5}{q^{\sqrt{n}/4}} + (\log_2\log_2{q}+1)(2^k+1)(2^{2m-2k}+1) \frac{2}{q^{n-3}}.
\]
In particular, for sufficiently large $m$, $Q(x,s) < \frac{1}{2}$ and $u(G) \geq 2$. Moreover, $Q(x,s) \to 0$ and $u(G) \to \infty$, as $m \to \infty$ or, for sufficiently large $m$, as $q \to \infty$.
\end{example}

\subsection{The generating graph} \label{ss:finite_graph}

In this section, we introduce a combinatorial object that gives a way to visualise the concepts we have introduced so far. The \emph{generating graph} of a group $G$, denoted $\Gamma(G)$, is the graph whose vertex set is $G \setminus 1$ and where two vertices $g, h \in G$ are adjacent if $\< g, h \> = G$. See Figure~\ref{fig:graph} for examples.

\begin{figure}[t]
\vspace{5mm}

\begin{center}
\begin{tikzpicture}[scale=0.42]
\node[white] (0) at (270:8cm) {};

\node[draw,circle,minimum size=0.8cm] (1) at (45 :5.5cm) {$ab$};
\node[draw,circle,minimum size=0.8cm] (2) at (135 :5.5cm) {$b$};
\node[draw,circle,minimum size=0.8cm] (3) at (225 :5.5cm) {$a^3b$};
\node[draw,circle,minimum size=0.8cm] (4) at (315 :5.5cm) {$a^2b$};

\node[draw,circle,minimum size=0.8cm] (5) at (0:1.5cm) {$a^3$};
\node[draw,circle,minimum size=0.8cm] (6) at (180:1.5cm) {$a$};

\node[draw,circle,minimum size=0.8cm] (8) at (270:6.5cm) {$a^2$};

\path (1) edge[-] (2);
\path (2) edge[-] (3);
\path (3) edge[-] (4);
\path (4) edge[-] (1);
\path (1) edge[-] (5);
\path (2) edge[-] (5);
\path (3) edge[-] (5);
\path (4) edge[-] (5);

\path (1) edge[-] (6);
\path (2) edge[-] (6);
\path (3) edge[-] (6);
\path (4) edge[-] (6);
\end{tikzpicture} \hspace{1cm}
\begin{tikzpicture}[scale=0.5,circle,inner sep=0.01cm,minimum size=0.66cm]

\node[draw] (1) at (90-360/11 :5cm) {\tiny{(1\,2)(3\,4)}};
\node[draw] (2) at (90 :5cm) {\tiny{(1\,3)(2\,4)}};
\node[draw] (3) at (90+360/11 :5cm) {\tiny{(1\,4)(2\,3)}};

\node[draw] (4) at (90+360/11*2: 5cm) {\tiny{(1\,2\,3)}};
\node[draw] (5) at (90+360/11*3: 5cm) {\tiny{(1\,3\,2)}};
\node[draw] (6) at (90+360/11*4: 5cm) {\tiny{(1\,2\,4)}};
\node[draw] (7) at (90+360/11*5: 5cm) {\tiny{(1\,4\,2)}};
\node[draw] (8) at (90+360/11*6: 5cm) {\tiny{(1\,3\,4)}};
\node[draw] (9) at (90+360/11*7: 5cm) {\tiny{(1\,4\,3)}};
\node[draw] (10) at (90+360/11*8: 5cm) {\tiny{(2\,3\,4)}};
\node[draw] (11) at (90+360/11*9: 5cm) {\tiny{(2\,4\,3)}};

\foreach \x in {1,...,3}
\foreach \y in {4,...,11}
{\path (\x) edge[-] (\y);}

\path (5)  edge[-] (6);
\path (7)  edge[-] (8);
\path (9)  edge[-] (10);
\path (11) edge[-] (4);

\foreach \x in {4,...,11}
\foreach \y in {4,...,11}
{
 \pgfmathtruncatemacro{\a}{\x-\y};
 \ifnum\a=0{}
 \else
 {
  \ifnum\a=1{}
  \else
  {
  \ifnum\a=-1{}
  \else
   {
    \path (\x) edge[-] (\y);
   }
   \fi
  }
  \fi
 }
 \fi
} 
\end{tikzpicture}
\end{center}
\caption{The generating graphs of $D_8 = \< a, b \mid a^4 = 1, b^2 = 1, a^b = a^{-1} \>$ and the alternating group $A_4$.} \label{fig:graph}
\end{figure}
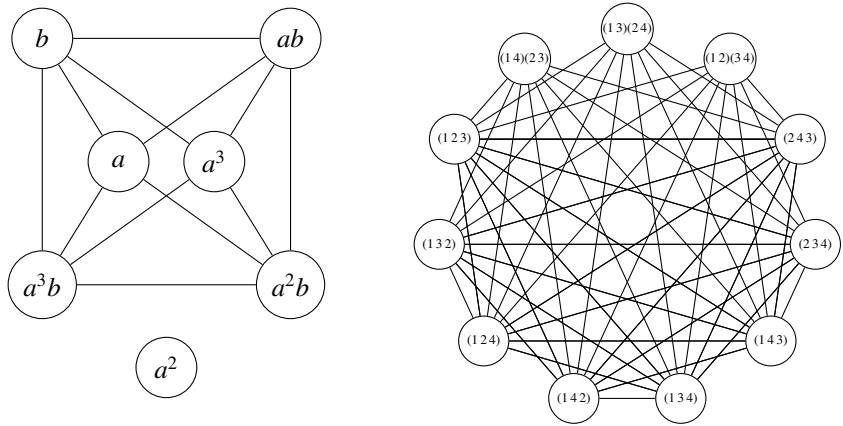

A question that immediately comes to mind is: when is $\Gamma(G)$ is connected? This question has a remarkably straightforward answer.

\begin{theorem} \label{thm:generating_graph}
Let $G$ be a finite group. Then the following are equivalent:
\begin{enumerate}
\item $\Gamma(G)$ has no isolated vertices
\item $\Gamma(G)$ is connected
\item $\Gamma(G)$ has diameter at most two
\item every proper quotient of $G$ is cyclic.
\end{enumerate}
\end{theorem}

Of course, Theorem~\ref{thm:generating_graph} is simply a reformulation of Theorem~\ref{thm:burness_guralnick_harper} due to Burness, Guralnick and Harper. Indeed, the generating graph gives an enlightening perspective on the concepts introduced so far:
\begin{enumerate}
\item $G$ is $\frac{3}{2}$-generated if and only if $\Gamma(G)$ has no isolated vertices
\item $s(G)$ is the greatest $k$ such that any $k$ vertices of $\Gamma(G)$ have a common neighbour, which means that $s(G) \geq 2$ if and only if $\mathrm{diam}(\Gamma(G)) \leq 2$
\item $\gamma_t(G)$ is the total domination number of $\Gamma(G)$: the least size of a set of vertices whose neighbours cover $\Gamma(G)$.
\end{enumerate}

Regarding (iii), the total domination number is a well studied graph invariant and was the inspiration for the group theoretic term (and the symbol $\gamma_t$ is the graph theoretic notation). Regarding (ii), as far as the author is aware, the graph invariant corresponding to the spread of a group, while natural, does not have a canonical name, but, inspired by the recent work on the spread of finite groups, some authors have started to use the term \emph{spread} for this graph invariant (that is, the \emph{spread} of a graph is the greatest $k$ such that any $k$ vertices have a common neighbour), see for example \cite[Section~2.5]{ref:Cameron22}. 

What are the connected components of $\Gamma(G)$ are in general? The following conjecture (first posed as a question in \cite{ref:CrestaniLucchini13-Israel}) proposes a straightforward answer.

\begin{conjecture} \label{con:generating_graph}
Let $G$ be a finite group. Then the graph obtained from $\Gamma(G)$ by removing the isolated vertices is connected.
\end{conjecture}

Conjecture~\ref{con:generating_graph} is known to be true if $G$ is soluble or characteristically simple by work of Crestani and Lucchini \cite{ref:CrestaniLucchini13-Israel,ref:CrestaniLucchini13-JAlgCombin} or if every proper quotient of $G$ is cyclic as a consequence of Theorem~\ref{thm:burness_guralnick_harper}. Otherwise, Conjecture~\ref{con:generating_graph} remains an intriguing open question about the generating sets of finite groups.

There is now a vast literature on the generating graph and surveying it is beyond the scope of this survey article, so we will make only a few remarks. The first paper to study $\Gamma(G)$ in its own right, and call it the \emph{generating graph}, was \cite{ref:LucchiniMaroti09Ischia} by Lucchini and Mar\'oti, who then studied various aspects of this graph in subsequent papers (for example, \cite{ref:BreuerGuralnickLucchiniMarotiNagy10,ref:LucchiniMaroti09}). However, $\Gamma(G)$ first appeared in the literature, indirectly, as a construction in a proof of Liebeck and Shalev in \cite{ref:LiebeckShalev96JAlg}. A major result of that paper is that there exist constants $c_1, c_2 > 0$ such that for all finite simple groups $G$, the probability that two randomly chosen elements generate $G$, denoted $P(G)$, satisfies
\begin{equation} \label{eq:prob}
1 - \frac{c_1}{m(G)} \leq P(G) \leq 1 - \frac{c_2}{m(G)}
\end{equation}
where $m(G)$ is the smallest index of a subgroup of $G$ (for example, $m(A_n)=n$). From this, Liebeck and Shalev deduce that there is a constant $c > 0$ such that every finite simple group $G$ contains at least $c \cdot m(G)$ elements that pairwise generate $G$. The proof of this corollary simply involves applying Tur\'an's theorem to $\Gamma(G)$, exploiting \eqref{eq:prob}. A couple of subsequent papers \cite{ref:Blackburn06,ref:BritnellEvseevGuralnickHolmesMaroti08} continued the study of cliques in $\Gamma(G)$, partly motivated by the observation that the largest size of a clique in $\Gamma(G)$, denoted $\mu(G)$, is a lower bound for the smallest number of proper subgroups whose union is $G$, denoted $\sigma(G)$. Returning to spread, it is easy to see that $s(G) < \mu(G) \leq \sigma(G)$, so these results give upper bounds on spread, which are otherwise difficult to find. Indeed, the best upper bounds for the smaller 14 sporadic groups (including the two where the spread is known exactly) were found in \cite{ref:BradleyHolmes07} by a clever refinement of this bound (for the larger 12 sporadic groups, different methods were used \cite{ref:Fairbarin12CA}).

\subsection{Applications} \label{ss:finite_app}

We conclude Section~\ref{s:finite} with three applications of spread. The first shows how $\frac{3}{2}$-generation naturally arises in a completely different context. The second highlights the benefit of studying spread, not just $\frac{3}{2}$-generation. The third moves beyond simple groups and applies the classification of finite $\frac{3}{2}$-generated groups.

\vspace{0.5\baselineskip}

\textbf{Application 1. Word maps. } From a word $w = w(x,y)$ in the free group $F_2$, we obtain a \emph{word map} $w\: G \times G \to G$ and we write $w(G) = \{ w(g,h) \mid g,h \in G\}$. 

Let $G$ be a nonabelian finite simple group. For natural choices of $w$, there has been substantial recent progress showing that $w(G) = G$. For example, solving the Ore Conjecture, Liebeck, O'Brien, Shalev and Tiep \cite{ref:LiebeckOBrienShalevTiep10} proved that $w = x^{-1}y^{-1}xy$ is surjective (that is, every element is a commutator). 

Now consider the converse question: which subsets $S \subseteq G$ arise as images of some word $w \in F_2$? Such a subset $S$ must satisfy $1 \in S$ (as $w(1,1)=1$) and $S^a = S$ for all $a \in \Aut(G)$ (as $w(x^a,y^a) = w(x,y)^a$), and Lubotzky \cite{ref:Lubotzky14} proved that these conditions are sufficient.

\begin{theorem} \label{thm:lubotzky}
Let $G$ be a finite simple group and let $S \subseteq G$. Then $S$ is the image of a word $w \in F_2$ if and only if $1 \in S$ and $S^a = S$ for all $a \in \Aut(G)$.
\end{theorem}

The proof of Theorem~\ref{thm:lubotzky} is short and fairly elementary, except it uses Theorem~\ref{thm:guralnick_kantor}. Indeed, Theorem~\ref{thm:guralnick_kantor} is the only way it depends on the CFSG.

\begin{proof}[Proof outline of Theorem~\ref{thm:lubotzky}]
Let $G^2 = \{ (a_i,b_i) \mid 1 \leq i \leq |G|^2 \}$ be ordered such that $\< a_i, b_i \> = G$ if and only if $i \leq \ell$. Fix the free group $F_2 = \< x, y \>$ and let $\p\: F_2 \to G^{|G|^2}$ be defined as $\p(x) = (a_1, \dots, a_{|G|^2})$ and $\p(y) = (b_1, \dots, b_{|G|^2})$. Let $z = (z_1, \dots, z_{|G|^2})$ where $z_i = a_i$ if $i \leq \ell$ and $a_i \in S$ and $z_i = 1$ otherwise.

By Theorem~\ref{thm:guralnick_kantor}, every nontrivial element of $G$ is contained in a generating pair, so, in particular, $\{ z_i \mid 1 \leq i \leq |G|^2 \} = S \cup \{1\} = S$. 

By an elementary argument, Lubotzky shows that $\p(F_2) = H \times K$ where $H$ and $K$ are the projections of $\p(F_2)$ onto the first $\ell$ factors of $G^{|G|^2}$ and the remaining $|G|^2-\ell$ factors, respectively. Moreover, using a theorem of Hall \cite{ref:Hall36}, $H$ is the subgroup of $G^\ell$ isomorphic to $G^{\ell/|\Aut(G)|}$ with the defining property that $(g_1, \dots, g_\ell) \in H$ if and only if for all $a \in \Aut(G)$ and $1 \leq i,j \leq \ell$ we have $g_i = g_j^a$ whenever $\p_i = \p_{\!j}\,a$. In particular, since $S^a = S$ for all $a \in \Aut(G)$, we deduce that $z \in \p(F_2)$. Therefore, there exists $w \in F_2$ such that for all $1 \leq i \leq |G|^2$ we have $w(a_i,b_i) = \p(w)_i = z_i$.

Combining the conclusions of the previous two paragraphs, we deduce that $w(G) = \{ w(a_i,b_i) \mid 1 \leq i \leq |G|^2 \} = \{ z_i \mid 1 \leq i \leq |G|^2 \} = S$.
\end{proof}

\vspace{0.5\baselineskip}

\textbf{Application 2. The product replacement graph. } For a positive integer $k$, the vertices of the product replacement graph $\Gamma_k(G)$ are the generating $k$-tuples of $G$, and the neighbours of $(x_1, \dots, x_i, \dots, x_k)$ in $\Gamma_k(G)$ are $(x_1, \dots, x_ix_j^\pm, \dots, x_k)$ and $(x_1, \dots, x_j^\pm x_i, \dots, x_k)$ for each $1 \leq i \neq j \leq k$. 

The product replacement graph arises in a number of contexts, most notably, the product replacement algorithm for computing random elements of $G$, which involves a random walk on $\Gamma_k(G)$, see \cite{ref:CellerLeedhamGreenMurrayNiemeyerOBrien95}. Thus, the connectedness of $\Gamma_k(G)$ is of particular interest. Specifically, Pak \cite[Question~2.1.33]{ref:Pak01} asked whether $\Gamma_k(G)$ is connected whenever $k$ is strictly greater than $d(G)$, the smallest size of a generating set for $G$. This question is open, even for finite simple groups where Wiegold conjectured that the answer is true. Nevertheless, the following lemma of Evans \cite[Lemma~2.8]{ref:Evans93} shows the usefulness of spread. (Here a generating tuple of $G$ is said to be \emph{redundant} if some proper subtuple also generates $G$.)

\begin{lemma}\label{lem:evans}
Let $k \geq 3$ and let $G$ be a group such that $s(G) \geq 2$. Then all of the redundant generating $k$-tuples are connected in $\Gamma_k(G)$.
\end{lemma}

\begin{proof}
Let $x = (x_1,\dots,x_k)$ and $y = (y_1,\dots,y_k)$ be two redundant generating $k$-tuples. By an elementary observation of Pak, it is sufficient to show that $x$ and $y$ are connected after permuting of the entries of $x$ and $y$. In particular, since $x$ and $y$ are redundant, we may assume that $\<x_1,\dots,x_{k-1}\> = \<y_1,\dots,y_{k-1}\> = G$ and also that $x_1 \neq 1 \neq y_2$. Since $s(G) \geq 2$, there exists $z \in G$ such that $\< x_1, z \>  = \< y_2, z \> = G$. We now make a series of connections. First, $x=x^{(1)}$ is connected to $x^{(2)} = (x_1,\dots,x_{k-1},z)$ as $\<x_1,\dots,x_{k-1}\> = G$. Next, $x^{(2)}$ is connected to $x^{(3)} = (x_1,y_2,\dots,y_{k-1},z)$ as $\<x_1,z\> = G$. Now, $x^{(3)}$ is connected to $x^{(4)} = (y_1,\dots,y_{k-1},z)$ as $\<y_2,z\> = G$. Finally, $x^{(4)}$ is connected to $(y_1,\dots,y_k) = y$  as $\<y_1,\dots,y_{k-1}\> = G$. This shows that $x$ is connected to $y$.
\end{proof}

Combining Lemma~\ref{lem:evans} with Theorem~\ref{thm:breuer_guralnick_kantor} shows that to prove Wiegold's conjecture, it suffices to show that for each finite simple group $G$ every irredundant generating $k$-tuple is connected in $\Gamma_k(G)$ to a redundant one.

As further evidence for the relevance of spread in this area, we note that Wiegold's original conjecture was (the a priori weaker claim) that a related graph $\Sigma_k(G)$ is connected for all finite simple groups $G$ and $k > d(G)$ (see \cite[Conjecture~2.5.4]{ref:Pak01}), but a short argument of Pak \cite[Proposition~2.5.13]{ref:Pak01} shows that $\Sigma_k(G)$ is connected if and only if $\Gamma_k(G)$ is connected, for all groups $G$ such that $s(G) \geq 2$, which by Theorem~\ref{thm:breuer_guralnick_kantor} includes all finite simple groups $G$.

\vspace{0.5\baselineskip}
 
\textbf{Application 3. The $\boldsymbol{\X}$-radical of a group. } In 1968, Thompson proved that a finite group is soluble if and only if all of its 2-generated subgroups are soluble \cite[Corollary~2]{ref:Thompson68}. The result follows from Thompson's classification of the finite insoluble groups all of whose proper subgroups are soluble, but, in 1995, Flavell gave a direct proof of this result \cite{ref:Flavell95}. Confirming a conjecture of Flavell \cite[Conjecture~B]{ref:Flavell01}, Guralnick, Kunyavsk\u{\i}i, Plotkin and Shalev proved the following result about the \emph{soluble radical} of $G$, written $R(G)$, which is the largest normal soluble subgroup of $G$ \cite[Theorem~1.1]{ref:GuralnickKunyavskiiPlotkinShalev06}. 

\begin{theorem} \label{thm:soluble_radical}
Let $G$ be a finite group. Then
\vspace{-2pt}
\[
R(G) = \{ x \in G \mid \text{$\<x,y\>$ is soluble for all $y \in G$} \}.
\]
\end{theorem}

The key (and only CFSG-dependent) element of the proof of Theorem~\ref{thm:soluble_radical} is a strong version Theorem~\ref{thm:guralnick_kantor} on the $\frac{3}{2}$-generation of finite simple groups. To paint a picture of how the $\frac{3}{2}$-generation of simple groups plays the starring role, we first present the analogue for simple Lie algebras \cite[Theorem~2.1]{ref:GuralnickKunyavskiiPlotkinShalev06}. Recall that the \emph{radical} of a Lie algebra $\g$, denoted $R(\g)$, is the largest soluble ideal of $\mathfrak{g}$.

\begin{theorem} \label{thm:soluble_radical_lie_algebras} 
Let $\g$ be a finite-dimensional Lie algebra over $\mathbb{C}$. Then
\vspace{-2pt}
\[
R(\g) = \{ x \in \g \mid \text{$\<x,y\>$ is soluble for all $y \in \g$} \}.
\]
\end{theorem}

\begin{proof}
Let $x \in \g$. First assume that $x \in R(\g)$. Let $y \in G$ and let $\mathfrak{h}$ be the smallest ideal of $\<x,y\>$ containing $x$. Now $\mathfrak{h}$ is soluble as it is a Lie subalgebra of $R(\g)$, and $\<x,y\>/\mathfrak{h}$ is soluble as it is $1$-dimensional, so $\<x,y\>$ is soluble.

Now assume that $\< x, y \>$ is soluble for all $y \in G$. We will prove that $x \in R(\g)$. For a contradiction, assume that $\g$ is a minimal counterexample (by dimension). Consider $\overline{\g} = \g/R(\g)$. Then $\< \overline{x}, \overline{y} \>$ is soluble for all $y \in \mathfrak{g}$, and $R(\overline{\g})$ is trivial, so $\overline{x} \not\in R(\overline{\g})$. Therefore, $R(\g)=0$, by the minimality of $\g$. This means that $\g$ is semisimple, so we may write $\g = \g_1 \oplus \cdots \oplus \g_k$ where $\g_1, \dots, \g_k$ are simple. Writing $x = (x_1, \dots, x_k)$, fix $1 \leq i \leq k$ such that $x_i \neq 0$. Then by Theorem~\ref{thm:ionescu}, there exists $y \in \g_i$ such that $\<x_i,y\> = \g_i$. In particular, $\g_i$ is a quotient of $\<x,y\>$, so $\<x,y\>$ is not soluble, which is a contradiction.
\end{proof}

Returning to groups, again using variants of Theorem~\ref{thm:guralnick_kantor}, Guralnick, Plotkin and Shalev set Theorem~\ref{thm:soluble_radical} in a more general context \cite[Theorem~6.1]{ref:GuralnickPlotkinShalev07}. Here we give a short proof of their result by applying Theorem~\ref{thm:burness_guralnick_harper}.

Let $\X$ be a class of finite groups that is closed under subgroups, quotients and extensions. The $\X$-radical of a group $G$, denoted $\X(G)$, is the largest normal $\X$-subgroup of $G$. For instance, $\X(G) = R(G)$ if $\X$ is the class of soluble groups.

\begin{theorem} \label{thm:x_radical}
Let $\X$ be a class of finite groups that is closed under subgroups, quotients and extensions. Then
\vspace{-2pt}
\[
\X(G) = \{ x \in G \mid \text{$\<x^{\<y\>}\>$ is an $\X$-group for all $y \in G$} \}.
\]
\end{theorem} 

\begin{corollary} \label{cor:x_radical}
Let $\X$ be a class of finite groups that is closed under subgroups, quotients and extensions. Then $G$ is an $\X$-group if and only if every $2$-generated subgroup of $G$ is an $\X$-group.
\end{corollary}

\begin{proof}
If $G$ is an $\X$-group, then every $2$-generated subgroup is. Conversely, if for all $x,y \in G$ the subgroup $\<x,y\>$ is an $\X$-group, then so is $\<x^{\<y\>}\>$, so, by Theorem~\ref{thm:x_radical}, $x \in \X(G)$, which shows $G = \X(G)$, which is an $\X$-group.
\end{proof}

\begin{corollary} \label{cor:x_radical_soluble}
Let $\X$ be a class of finite groups that is closed under subgroups, quotients and extensions. Assume that $\X$ contains all soluble groups. Then
\vspace{-2pt}
\[
\X(G) = \{ x \in G \mid \text{$\<x, y\>$ is an $\X$-group for all $y \in G$} \}.
\]
\end{corollary}

\begin{proof}
Let $x \in G$. If for all $y \in G$, $\<x,y\>$ is an $\X$-group, then so is $\<x^{\<y\>}\>$, so, by Theorem~\ref{thm:x_radical}, $x \in \X(G)$. Conversely, if $x \in \X(G)$, then $\<x^{\<y\>}\> \leq \<x^G\> \leq \X(G)$ is an $\X$-group, so $\<x,y\>$, an extension of $\<x^{\<y\>}\>$ by a cyclic group, is an $\X$-group.
\end{proof}

\begin{proof}[Proof of Theorem~\ref{thm:x_radical}]
Let $x \in G$. First assume that $x \in \X(G)$. For all $y \in G$, we have $\<x^{\<y\>}\> \leq \<x^G\> \leq \X(G)$, so $\<x^{\<y\>}\>$ is an $\X$-group.

Now assume that $\< x^{\<y\>} \>$ is an $\X$-group for all $y \in G$. We will prove that $x \in \X(G)$. For a contradiction, assume that $G$ is a minimal counterexample. Consider $\overline{G} = G/\X(G)$. Then $\< \overline{x}^{\<\overline{y}\>} \>$ is an $\X$-group for all $y \in G$ (as $\X$ is closed under quotients), and $\X(\overline{G})$ is trivial (as $\X$ is closed under extensions), so $\overline{x} \not\in \X(\overline{G})$. Therefore, $\X(G)=1$, by the minimality of $G$. Now consider $H = \<x^G\>$. Then $\<x^{\<h\>}\>$ is an $\X$-group for all $h \in H$, and $\X(H) \leq \X(G)$ (as $\X(H)$ is characteristic in $H$ so normal in $G$), so $x \not\in \X(H)$. Therefore, $\<x^G\> = G$, by the minimality of $G$. Consider a power $x'$ of $x$ of prime order. Then $\< x'^{\<y\>} \>$ is an $\X$-group for all $y \in G$ (as $\X$ is closed under subgroups), and $x' \not\in 1 = \X(G)$. Therefore, it suffices to consider the case where $x$ has prime order. 

Let $N$ be a minimal normal subgroup of $G$ and write $N = T^k$ where $T$ is simple. Observe that $N$, or equivalently $T$, is not an $\X$-group, since $\X(G)=1$.

Suppose that $x \in N$. Then $G = N$ since $\<x^G\> = G$. In particular, $k=1$, so, by Theorem~\ref{thm:guralnick_kantor}, there exists $y \in G$ such that $\<x,y\>=G$, which is not an $\X$-group, so $\<x^{\<y\>}\>$ is not an $\X$-group either: a contradiction. Therefore, $x \not\in N$.

Suppose that $x$ centralises $N$. Then $N$ is central since $\<x^G\> = G$, so $N \cong C_p$ for a prime $p$. In $\widetilde{G} = G/N$, $\widetilde{x}$ is nontrivial as $x \not\in N$ and $\<\widetilde{x}^{\<\widetilde{y}\>}\>$ is an $\X$-group for all $y \in G$ (as $\<x^{\<y\>}\>$ is). The minimality of $G$ means $\widetilde{x} \in \X(\widetilde{G})$, but $\<\widetilde{x}^{\widetilde{G}}\> = \widetilde{G}$ (as $\<x^G\> = G$), so $\widetilde{G}$ is an $\X$-group. Since $N \cong C_p$ is not an $\X$-group, $p$ does not divide $|\widetilde{G}|$. By the Schur--Zassenhaus Theorem, $G = N \times H$ for some $H \cong \widetilde{G}$. Since $N \cong C_p$ is not an $\X$-group and $\<x\>$ is an $\X$-group, $p$ does not divide $|x|$, so $x \in H$, contradicting $\< x^G \> = G$. Therefore, $x$ acts nontrivially on $N$.

Suppose that $N$ is abelian. As $x$ acts nontrivially on $N$, there is $n \in N$ with $[x,n] \neq 1$. Now $\< [x,n] \>$ is isomorphic to $T$, as $[x,n] \in N$, so it is not an $\X$-group, implying that $\< x^{\<n\>} \>$ is not an $\X$-group either. Therefore, $N$ is nonabelian.

Now $x$ permutes the $k$ factors of $N \cong T^k$ and let $M \cong T^l$ be a nontrivial subgroup of $N$ whose factors are permuted transitively by $x$. If $l=1$, then $\< M, x \>$ is almost simple, and if $l > 1$, then, recalling that $x$ has prime order, $\<x\> \cong C_l$ acts regularly on the factors of $M$. In either case, $M$ is the unique minimal normal subgroup of $\<M,x\>$, so every proper quotient of $\<M,x\>$ is cyclic. Therefore, by Theorem~\ref{thm:burness_guralnick_harper}, there exists $m \in \< M, x \>$ such that $\< m, x \> = \< M, x\>$. Moreover, $\<x^{\<m\>}\>$, being a normal subgroup of $\< M, x \>$ containing $x$, is also $\< M, x\>$. However, $\<x^{\<m\>}\> = \< M, x \>$ is not an $\X$-group since the subgroup $T$ is not an $\X$-group. This contradiction completes the proof.
\end{proof}

\section{Infinite Groups} \label{s:infinite}

\subsection{Generating infinite groups} \label{ss:infinite_intro}

We now turn to infinite groups and their generating pairs. Do the results from Section~\ref{s:finite} on the spread of finite groups extend to infinite groups? Let us recall that the motivating theorem for finite groups is the landmark result that every finite simple group is $2$-generated. This result is easily seen to be false when the assumption of finiteness is removed. For example, the alternating group $\Alt(\Int)$ is simple but is not finitely generated since every finite subset of $\Alt(\Int)$ is supported on finitely many points and therefore generates a finite subgroup. The problem persists even if we restrict to finitely generated simple groups. Answering a question of Wiegold in the Kourovka Notebook \cite[Problem~6.44]{ref:Kourovka22}, in 1982, Guba constructed a finitely generated infinite simple group that is not 2-generated (in fact, the group constructed has the property that every 2-generated subgroup is free) \cite{ref:Guba82}. More recently, Osin and Thom, by studying the $\ell^2$-Betti number of groups, proved that for every $k \geq 2$ there exists an infinite simple group that is $k$-generated but not $(k-1)$-generated \cite[Corollary~1.2]{ref:OsinThom13}. 

With these results in mind, it makes sense to focus on $2$-generated groups and ask whether the results about the spread of finite $2$-generated groups extend to general $2$-generated groups. Recall that Theorem~\ref{thm:burness_guralnick_harper} gives a characterisation of the finite $\frac{3}{2}$-generated groups: a finite group $G$ is $\frac{3}{2}$-generated if and only if every proper quotient of $G$ is cyclic. In particular, every finite simple group is $\frac{3}{2}$-generated. The following example due to Cox in 2022 \cite{ref:Cox22}, highlights that this characterisation does not extend to general $2$-generated groups (that is, there exists a infinite $2$-generated group $G$ that is not $\frac{3}{2}$-generated but for which every proper quotient of $G$ is cyclic).

\begin{example} \label{ex:cox}
For each positive integer $n$, let $G_n$ be the subgroup of $\Sym(\Int)$ defined as $\< \Alt(\Int), t^n \>$ where $t\: \Int \to \Int$ is the translation $x \mapsto x+1$. 

It is straightforward to show that $\Alt(\Int)$ is the unique minimal normal subgroup of $G_n$, so every proper quotient of $G$ is cyclic. 

In addition, $G_n$ is $2$-generated. Indeed, by \cite[Lemma~3.7]{ref:Cox22}, $G_n = \< a_n, t^n \>$ for $a_n = \prod_{i=0}^{m+1} x_i^{t^{3ni}}$ where $x_0, \dots, x_{m+1} \in A_{3n}$ satisfy $x_0 = (1 \, 3)$, $x_{m+1} = (2 \, 3)$ and $\{ (1 \, 2 \, 3)^{x_1}, \dots, (1 \, 2 \, 3)^{x_m} \} = (1 \, 2 \, 3)^{A_{3n}}$. To see this, we note that $[a_n^{t^{-3n(m+1)}},a_n] = [x_{m+1},x_0] = (1 \, 2 \, 3)$ and $(1 \, 2 \, 3)^{a_n^{t^{-3ni}}} = (1 \, 2 \, 3)^{x_i}$, so $\< a_n, t^n \> \geq \< (1 \, 2 \, 3)^{A_{3n}}, t^n \> = \< A_{3n}, t^n \>$, which is simply $\< \Alt(\Int), t^n \>$ (compare with Lemma~\ref{lem:covering_symmetric} below). 

However, in \cite[Theorem~4.1]{ref:Cox22}, Cox proves that if $n \geq 3$, then $(1 \, 2 \, 3)$ is not contained in a generating pair for $G_n$, so $G_n$ gives an example of a $2$-generated group all of whose proper quotients are cyclic but which is not $\frac{3}{2}$-generated. To simplify the proof, we will assume that $n \geq 4$. Let $g \in G_n$. If $g \in \Alt(\Int)$, then $\< (1 \, 2 \, 3), g \> \leq \Alt(\Int) < G_n$. Now assume that $g \not\in \Alt(\Int)$, so $\<g\> = \<ht^k\>$ where $h \in \Alt(\Int)$ and $k \geq 4$. It is easy to show that $g$ has exactly $k$ infinite orbits $O_1, \dots, O_k$ (indeed, if $\mathrm{supp}(h) \subseteq [a,b]$, then we quickly see that for a suitable permutation $\pi \in \Sm{k}$, we can find $k$ orbits $O_1, \dots, O_k$ of $g$ satisfying $O_i \setminus [a,b] = \{ x > b \mid x \equiv i \mod{k} \} \cup \{ x < a \mid x \equiv i\pi \mod{k} \}$). Since $k \geq 4$, we can fix $i$ such that $O_i \cap \{1, 2, 3\} = \emptyset$, so $O_i$ is an orbit of $\< (1 \, 2 \, 3), g \>$, which implies that $\< (1 \, 2 \, 3), g \> \neq G_n$ in this case too.

In contrast, in \cite[Theorem~6.1]{ref:Cox22}, Cox shows that $G_1$ and $G_2$ are $\frac{3}{2}$-generated, and, in fact, $2 \leq u(G_i) \leq s(G_i) \leq 9$ for $i \in \{1,2\}$.
\end{example}

The groups in Example~\ref{ex:cox} are not simple, so the following question remains.

\begin{question} \label{que:infinte_0}
Does there exist a $2$-generated simple group $G$ with $s(G)=0$?
\end{question}

Recall that for finite groups $G$, Theorem~\ref{thm:burness_guralnick_harper} also establishes that $s(G) \geq 1$ if and only if $s(G) \geq 2$, so there are no finite groups $G$ satisfying $s(G)=1$. This raises the following question.

\begin{question} \label{que:infinite_1}
Does there exist a $2$-generated simple group $G$ with $s(G)=1$?
\end{question}

There is a clear difference between generating finite and infinite groups, and straightforward analogues of the theorems for finite groups do not hold for infinite groups. Nevertheless, do the results on the spread of finite simple groups extend to important classes of infinite simple groups? The investigation of the infinite simple groups of Richard Thompson (and their many generalisations) in Sections~\ref{ss:infinite_thompson_introduction}--\ref{ss:infinite_thompson_t} demonstrates that the answer is a resounding yes! However, before turning to these infinite simple groups, in Section~\ref{ss:infinite_soluble}, we look at the other important special case we considered for finite groups: soluble groups.

\subsection{Soluble groups} \label{ss:infinite_soluble}

In the opening to Section~\ref{ss:finite_bgh}, we noted that when Brenner and Wiegold introduced the notion of spread, they proved that for a finite soluble group $G$, we have $s(G) \geq 1$ if and only if $s(G) \geq 2$ if and only if every proper quotient of $G$ is cyclic (see Theorem~\ref{thm:brenner_wiegold}). By Theorem~\ref{thm:breuer_guralnick_kantor}, ``soluble'' can be removed from the hypothesis (while keeping ``finite''). The following theorem establishes that ``finite'' can be removed from the hypothesis (while keeping ``soluble''), in a very strong sense. (Theorem~\ref{thm:infinite_soluble} is due to the author, and this is the first appearance of it in the literature.)

\begin{theorem} \label{thm:infinite_soluble}
Let $G$ be an infinite soluble group such that every proper quotient is cyclic. Then $G$ is cyclic.
\end{theorem}

\begin{proof}
It suffices to show that $G$ is abelian, because an infinite abelian group where every proper quotient is cyclic is itself cyclic. For a contradiction, suppose that $G$ is nonabelian. If $1 \neq N \leqn G$, then $G/N$ is cyclic, so $G' \leq N$. Therefore, $G'$ is the unique minimal normal subgroup of $G$. In particular, $G''$ is $G'$ or $1$, but $G$ is soluble, so $G'' = 1$, which implies that $G'$ is abelian. Therefore, $G'$ is an abelian characteristically simple group, so it is isomorphic to the additive group of a vector space $V$ over a field $F$, and we may assume that $F = \F_p$ or $F=\mathbb{Q}$.

Let $g \in G$ such that $G/V = \<Vg\>$ and write $H = \<g\>$, so $G = VH$. Observe that $Z(G)=1$, for otherwise $G/Z(G)$ is cyclic, so $G$ is abelian, a contradiction. Now, if $g^i \in V$, then $g^i \in Z(G) = 1$, so $V \cap H = 1$. Hence, $G$ is a semidirect product $V{:}H$. For all nontrivial $v \in V$, we have $V =  \< v^G \>$ since $V$ is a minimal normal subgroup and $\< v^G \>  = \< v^H \>$ since $V$ is abelian. Therefore, $V$ is an irreducible $FH$-module, so $V$ is finite-dimensional since $H$ is cyclic. (To see this, suppose that $V$ is infinite-dimensional, so $H$ is infinite. We give a proper nonzero submodule $U$, contradicting the irreducibility of $V$. For $0 \neq v \in V$, either $\{ vg^i \mid i \in \mathbb{Z} \}$ is linearly independent, and $U$ is the kernel of $\sum_{i \in \mathbb{Z}} a_ivg^i  \mapsto \sum_{i \in \mathbb{Z}} a_i$, or for some $u = vg^i$ we have $a_0 u + a_1 ug + \cdots + a_k ug^k = 0$ and $U = \< u, ug, \dots, ug^{k-1} \>$.) If $g^i \in C_G(V)$, then $g^i \in Z(G) = 1$, so $V$ is a faithful $FH$-module. In particular, if $F$ is finite, then so is $G = F^n{:}H \leq F^n{:}\GL_n(F)$, so we must have $F = \mathbb{Q}$.  

Let $\chi = X^n + a_{n-1}X^{n-1} + \cdots + a_1X + a_0 \in \mathbb{Q}[X]$ be the characteristic polynomial of $g$, and let $(e_1,\dots,e_n)$ be a basis for $V$ with respect to which the matrix $A$ of $g$ is the companion matrix of $\chi$. Let $P$ be the set of prime divisors appearing in the reduced forms of $a_0, \dots, a_{n-1}$ and note that $P$ is finite. For all $i \in \Int$, write $e_1A^i$ as a linear combination $\lambda_{i1}e_1 + \cdots + \lambda_{in}e_n$. Any prime that divides the denominator of the reduced form of one of the $\lambda_{ij}$ is contained in $P$. Hence, only finitely many primes appear in the denominators of the reduced forms of any element in the subgroup $N$ generated by $\{ e_1A^i \mid i \in \Int \}$. Since $N$ is $\<e_1^G\>$, it is a proper nontrivial subgroup of $V$ that is normal in $G$, which contradicts $V$ being a minimal normal subgroup of $G$. Therefore, $G$ is abelian and so cyclic.
\end{proof}

With a much shorter proof, one can obtain an analogous result for the class of residually finite groups (this was observed by Cox in \cite[Lemma~1.1]{ref:Cox22}).

\begin{theorem} \label{thm:infinite_residually_finite}
Let $G$ be an infinite residually finite group such that every proper quotient is cyclic. Then $G$ is cyclic.
\end{theorem}

\begin{proof}
Suppose that $G$ is nonabelian. Fix $x,y \in G$ with $[x,y] \neq 1$. Since $G$ is residually finite, $G$ has a finite index normal subgroup $N$ such that $[Nx,Ny]$ is nontrivial in $G/N$ (so, $Nx$ and $Ny$ are nontrivial in $G/N$). Since $G$ is infinite and $N$ has finite index, we know that $N$ is nontrivial, so $G/N$ is cyclic, which contradicts $G/N$ being nonabelian. Therefore, $G$ is abelian and hence cyclic.
\end{proof}

\subsection{Thompson's groups: an introduction} \label{ss:infinite_thompson_introduction}

In 1965, Richard Thompson introduced three finitely generated infinite groups $F < T < V$ \cite{ref:Thompson65}. Among other interesting properties of these groups, $V$ and $T$ were the first known examples of finitely presented infinite simple groups and for 35 years (until the work of Burger and Mozes \cite{ref:BurgerMozes00}) all known examples of such groups were closely related to $T$ and $V$. For an indication of other interesting properties of these groups, we record that $F$ is finitely presented yet it contains a copy of $F \times F$ and is an HNN extension of itself. Moreover, $F$ has exponential growth but contains no nonabelian free groups, and one of the most famous open questions in geometric group theory is whether $F$ is amenable \cite{ref:Cleary17}. However, these three groups not only raise interesting group theoretic questions, but they have also played a role in a whole range of mathematical areas such as the word problem for groups, homotopy theory and dynamical systems (see \cite{ref:Dydak77_2,ref:GhysSergiescu87,ref:Thompson80} for example). We refer the reader to Canon, Floyd and Parry's introduction to these groups \cite{ref:CannonFloydParry96}. 

An appealing feature of Thompson's groups is that they admit concrete representations as transformation groups, which we outline now.  Let $X = \{0,1\}^*$ be the set of all finite words over $\{0,1\}$, and let $\C = \{0,1\}^\Nat$ be \emph{Cantor space}, the set of all infinite sequences over $\{0,1\}$ with the usual topology. For $u \in X$, we write $u\C = \{uw \mid w \in \C \}$, and we say a finite set $A \subseteq X$ is a \emph{basis} of $\C$ if $\{ u\C \mid u \in A \}$ is a partition of $\C$. Thompson's group $V$ is the group of homeomorphisms $g \in \Homeo(\C)$ for which there exists a \emph{basis pair}, namely a bijection $\sigma\: A \mapsto B$ between two bases $A$ and $B$ of $\C$ such that $(uw)g = (u\sigma)w$ for all $u \in A$ and all $w \in \C$. In other words, $V$ is the group of homeomorphisms of $\C$ that act by prefix substitutions. For instance, $c\:(00,01,1) \mapsto (11,0,10)$ is an element of $V$ that, for example, maps $01010101\dots$ to $0010101\dots$. The selfsimilarity of $\C$ means that there is not a unique choice of basis pair; indeed, by subdividing $\C$ further $c$ is also represented by $(000, 001, 01, 1) \mapsto (110, 111, 0, 10)$. By identifying elements of $X$ with vertices of the infinite binary rooted tree, we can represent bases as binary rooted trees and elements of $V$ by the familiar \emph{tree pairs}, as shown in Figure~\ref{fig:elements}.

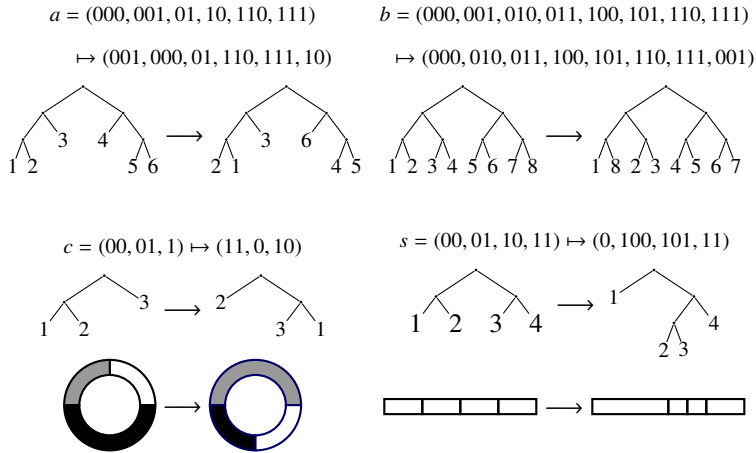
\begin{figure}

\begin{minipage}{0.42\textwidth}
\begin{gather*}
\text{\footnotesize $a = (000,001,01,10,110,111)$} \\
\text{\footnotesize \qquad $\mapsto (001,000,01,110,111,10)$} \\[1pt]
  \begin{tikzpicture}[
      scale=0.5,
      font=\footnotesize,
      inner sep=0pt,
      baseline=-30pt,
      level distance=20pt,
      level 1/.style={sibling distance=60pt},
      level 2/.style={sibling distance=30pt},
      level 3/.style={sibling distance=15pt}
    ]
    \node (root) [circle,fill] {}
    child {node (0) [circle,fill] {}
      child {node (00) [circle,fill] {}
        child {node (000) {1}}
        child {node (001) {2}}}
      child {node (01) {3}}}
    child {node (1) [circle,fill] {}
      child {node (10) {4}}
      child {node (11) [circle,fill] {}
        child {node (110) {5}}
        child {node (111) {6}}}};
  \end{tikzpicture}
  \raisebox{3mm}{ $\longrightarrow$ }
  \begin{tikzpicture}[
      scale=0.5,
      font=\footnotesize,
      inner sep=0pt,
      baseline=-30pt,
      level distance=20pt,
      level 1/.style={sibling distance=60pt},
      level 2/.style={sibling distance=30pt},
      level 3/.style={sibling distance=15pt}
    ]
    \node (root) [circle,fill] {}
    child {node (0) [circle,fill] {}
      child {node (00) [circle,fill] {}
        child {node (000) {2}}
        child {node (001) {1}}}
      child {node (01) {3}}}
    child {node (1) [circle,fill] {}
      child {node (10) {6}}
      child {node (11) [circle,fill] {}
        child {node (110) {4}}
        child {node (111) {5}}}};
  \end{tikzpicture}
\end{gather*}
\end{minipage}
\begin{minipage}{0.48\textwidth}
\begin{gather*}
\text{\footnotesize $b = (000,001,010,011,100,101,110,111)$} \\
\text{\footnotesize \quad $\mapsto (000,010,011,100,101,110,111,001)$} \\[1pt]
  \begin{tikzpicture}[
        scale=0.5,
        font=\footnotesize,
        inner sep=0pt,
        baseline=-30pt,
        level distance=20pt,
        level 1/.style={sibling distance=60pt},
        level 2/.style={sibling distance=30pt},
        level 3/.style={sibling distance=15pt}
      ]
      \node (root) [circle,fill] {}
      child {node (0) [circle,fill] {}
        child {node (00) [circle,fill] {}
          child {node (000) {1}}
          child {node (001) {2}}}
        child {node (01) [circle,fill] {}
          child {node (010) {3}}
          child {node (011) {4}}}}
      child {node (1) [circle,fill] {}
        child {node (10) [circle,fill] {}
          child {node (100) {5}}
          child {node (101) {6}}}
        child {node (11) [circle,fill] {}
          child {node (110) {7}}
          child {node (111) {8}}}};
  \end{tikzpicture}
  \raisebox{3mm}{ $\longrightarrow$ }
  \begin{tikzpicture}[
        scale=0.5,
        font=\footnotesize,
        inner sep=0pt,
        baseline=-30pt,
        level distance=20pt,
        level 1/.style={sibling distance=60pt},
        level 2/.style={sibling distance=30pt},
        level 3/.style={sibling distance=15pt}
      ]
      \node (root) [circle,fill] {}
      child {node (0) [circle,fill] {}
        child {node (00) [circle,fill] {}
          child {node (000) {1}}
          child {node (001) {8}}}
        child {node (01) [circle,fill] {}
          child {node (010) {2}}
          child {node (011) {3}}}}
      child {node (1) [circle,fill] {}
        child {node (10) [circle,fill] {}
          child {node (100) {4}}
          child {node (101) {5}}}
        child {node (11) [circle,fill] {}
          child {node (110) {6}}
          child {node (111) {7}}}};
  \end{tikzpicture}
\end{gather*}
\end{minipage}
\vspace{2.5mm}

\begin{minipage}{0.42\textwidth}
\begin{gather*}
\text{\footnotesize $c = (00,01,1) \mapsto (11,0,10)$} \\[1pt]
  \begin{tikzpicture}[
        scale=0.5,
        font=\footnotesize,
        inner sep=0pt,
        baseline=-30pt,
        level distance=20pt,
        level 1/.style={sibling distance=60pt},
        level 2/.style={sibling distance=30pt}
      ]
      \node (root) [circle,fill] {}
      child {node (0) [circle,fill] {}
        child {node (00) {1}}
        child {node (01) {2}}}
      child {node (1) {3}};
  \end{tikzpicture}
  \raisebox{5mm}{ \ $\longrightarrow$ \ }
  \begin{tikzpicture}[
        scale=0.5,
        font=\footnotesize,
        inner sep=0pt,
        baseline=-30pt,
        level distance=20pt,
        level 1/.style={sibling distance=60pt},
        level 2/.style={sibling distance=30pt}
      ]
      \node (root) [circle,fill] {}
      child {node (0) {2}}
      child {node (1) [circle,fill] {}
        child {node (10) {3}}
        child {node (11) {1}}};
  \end{tikzpicture} \\[-2mm]
  \begin{tikzpicture}[scale=2,thick]
  \draw [fill=white]         (  0:0.3) -- (  0:0.2) arc (  0: 90:0.2) -- ( 90:0.3) arc ( 90:  0:0.3);
  \draw [fill=black!40]       ( 90:0.3) -- ( 90:0.2) arc ( 90:180:0.2) -- (180:0.3) arc (180: 90:0.3);
  \draw [fill=black] (180:0.3) -- (180:0.2) arc (180:360:0.2) -- (360:0.3) arc (360:180:0.3);
  \end{tikzpicture}
  \raisebox{5mm}{ $\longrightarrow$ }
  \begin{tikzpicture}[scale=2,thick,draw=blue!40!black]
  \draw [fill=black!40]       (  0:0.3) -- (  0:0.2) arc (  0:180:0.2) -- (180:0.3) arc (180:  0:0.3);
  \draw [fill=black] (180:0.3) -- (180:0.2) arc (180:270:0.2) -- (270:0.3) arc (270:180:0.3);
  \draw [fill=white]         (270:0.3) -- (270:0.2) arc (270:360:0.2) -- (360:0.3) arc (360:270:0.3);
  \end{tikzpicture}
\end{gather*}
\end{minipage}
\begin{minipage}{0.48\textwidth}
\begin{gather*}
\text{\footnotesize $s = (00,01,10,11) \mapsto (0,100,101,11)$} \\[1pt]
\begin{tikzpicture}[
      scale=0.5,
      inner sep=0pt,
      baseline=-30pt,
      level distance=20pt,
      level 1/.style={sibling distance=60pt},
      level 2/.style={sibling distance=30pt},
      level 3/.style={sibling distance=15pt}
    ]
    \node (root) [circle,fill] {}
    child {node (0) [circle,fill] {}
      child {node (00) {1}}
      child {node (01) {2}}}
    child {node (1) [circle,fill] {}
      child {node (10) {3}}
      child {node (11) {4}}};
  \end{tikzpicture}
  \raisebox{5mm}{ \ $\longrightarrow$ \ }
  \begin{tikzpicture}[
      scale=0.5,
      font=\footnotesize,
      inner sep=0pt,
      baseline=-30pt,
      level distance=20pt,
      level 1/.style={sibling distance=60pt},
      level 2/.style={sibling distance=30pt},
      level 3/.style={sibling distance=15pt}
    ]
    \node (root) [circle,fill] {}
    child {node (0) {1}}
    child {node (1) [circle,fill] {}
      child {node (10) [circle,fill] {}
        child {node (100) {2}}
        child {node (101) {3}}}
      child {node (11) {4}}};
  \end{tikzpicture} \\[3.1mm]
  \begin{tikzpicture}[scale=2,thick]
  \draw (0,   0) rectangle (0.25,0.1);
  \draw (0.25,0) rectangle (0.5, 0.1);
  \draw (0.5, 0) rectangle (0.75,0.1);
  \draw (0.75,0) rectangle (1,   0.1);
  \end{tikzpicture}
  \longrightarrow
  \begin{tikzpicture}[scale=2,thick]
  \draw (0,    0) rectangle (0.5,  0.1);
  \draw (0.5,  0) rectangle (0.625,0.1);
  \draw (0.625,0) rectangle (0.75, 0.1);
  \draw (0.75, 0) rectangle (1,    0.1);
  \end{tikzpicture}
\end{gather*}
\vspace{-0.2mm}
\end{minipage}

\caption{Four elements of Thompson's group $V$.} \label{fig:elements}
\end{figure}

A motivating perspective in the study of generating sets of Thompson's groups is that $V$ combines the selfsimilarity of the Cantor space with permutations from the symmetric group. Indeed, to $g \in V$ we may associate (not uniquely) a permutation as follows. Let $\sigma\:A \to B$ be a basis pair for $g$ and write $A = \{ a_1, \dots, a_n \}$ and $B = \{ b_1, \dots, b_n \}$ where $a_1 < \dots < a_n$ and $b_1 < \dots < b_n$ in the lexicographic order. Then the permutation associated to $g$ is the element $\pi_g \in \Sm{n}$ satisfying $a_ig = b_{i\pi_g}$. For the elements in  Figure~\ref{fig:elements}, for instance,
\[
\pi_a = (1 \, 2)(4 \, 5 \, 6), \quad \pi_b = (2 \, 3 \, 4 \, 5 \, 6 \, 7 \, 8), \quad \pi_c = (1 \, 2 \, 3), \quad \pi_s = 1.
\]
This perspective gives an easy way to define $F$ and $T$. Thompson's groups $F$ and $T$ are the subgroups of $V$ of elements whose associated permutation is trivial and cyclic, respectively (this is well defined). Clearly $F < T < V$ and, referring to Figure~\ref{fig:elements}, we see that $s \in F$, $c \in T \setminus F$ and $a,b \in V \setminus T$.

Given a binary word $u$, we can associate a subset $I_u$ of the unit interval $[0,1]$ (or unit circle $\S^1$) inductively as follows: the empty word corresponds to $(0,1)$ and for any binary word $u$, the words $u0$ and $u1$ correspond to the open left and right halves of $u$. In this way, a basis for $\C$ can be interpreted as a sequence of disjoint open intervals whose closures cover $[0,1]$ (or $\S^1$), and an element of $V$, as a basis pair, $\sigma\:\{a_1, \dots, a_n\}\to\{b_1, \dots, b_n\}$ defines a bijection $g\:[0,1] \to [0,1]$ (or $g\:\S^1 \to \S^1$) by specifying that $I_{a_i}g = I_{b_i}$ and $g|_{I_{a_i}}$ is affine for all $1 \leq i \leq n$. Under this correspondence, $F$ is a group of piecewise linear homeomorphisms of $[0,1]$ and $T$ is a group of piecewise linear homeomorphisms of $\S^1$. See Figure~\ref{fig:elements} for some examples.

\subsection{\boldmath Thompson's group $V$} \label{ss:infinite_thompson_v}

The final three sections of this survey address generating sets for Thompson's groups, which have seen lots of very recent progress. We begin with $V$.

The group $V$ is $2$-generated. Indeed, referring to Figure~\ref{fig:elements}, $V = \< a,b \>$.  Given this infinite $2$-generated simple group $V$, we are naturally led to ask: is $V$ $\frac{3}{2}$-generated? An answer was given by Donoven and Harper in 2020 \cite{ref:DonovenHarper20}.

\begin{theorem} \label{thm:donoven_harper}
Thompson's group $V$ is $\frac{3}{2}$-generated.
\end{theorem}

Theorem~\ref{thm:donoven_harper} gave the first example of a noncyclic infinite $\frac{3}{2}$-generated group, other than the pathological \emph{Tarski monsters}: the infinite groups whose only proper nontrivial subgroups have order $p$ for a fixed prime $p$, which are clearly simple and $\frac{3}{2}$-generated and were proved to exist for all $p > 10^{75}$ by Olshanskii in \cite{ref:Olshanskii80}. (Note that the groups $G_1$ and $G_2$ in Example~\ref{ex:cox} were found later by Cox in \cite{ref:Cox22}, motivated by a question posed in \cite{ref:DonovenHarper20}.) In particular, $V$ was the first finitely presented example of a noncyclic infinite $\frac{3}{2}$-generated group.

\vspace{0.5\baselineskip}

\textbf{Methods. A parallel with symmetric groups. } Let us write $G = \Sm{n}$ and $\Omega = \{ 1, \dots, n \}$. It is well known that $G$ is generated by the set of transpositions $\{ (i \, j) \mid \text{distinct $i, j \in \Omega$} \}$ which yields a natural presentation for $G$, namely 
\begin{equation} \label{eq:presentation_symmetric}
\< t_{i,j} \mid t_{i,j}^2, \ t_{i,j}^{t_{k,l}} = t_{i (k \, l),j(k \, l)} \>
\end{equation}
Moreover, using the fact that $G$ is generated by transpositions, we obtain the following \emph{covering lemma} (here $G_{[A]}$ is the subgroup of $G$ supported on $A \subseteq \Omega$).

\begin{lemma} \label{lem:covering_symmetric}
Let $A_1, \dots, A_k \subseteq \Omega$ satisfy $A_i \cap A_{i+1} \neq \emptyset$ and $\bigcup_{i=1}^k{A_i} = \Omega$. Then $G_{[A_i]} \cong \Sm{|A_i|}$ for all $i$, and $G = \< G_{[A_1]}, \dots, G_{[A_k]} \>$. 
\end{lemma}

These results for $G = \Sm{n}$ have analogues for $V$. Here, \emph{transpositions} are elements of the following form: for $u, v \in X$ such that the corresponding subsets of $\C$ are disjoint, we write $(u \, v)$ for the element given as $(u,v,w_1, \dots, w_k) \mapsto (v,u,w_1, \dots, w_k)$ where $\{ u, v, w_1, \dots, w_k\}$ is a basis for $\C$. Brin \cite{ref:Brin04} proved that $V$ is generated by $\{ (u \, v) \mid \text{disjoint $u,v \in X$} \}$. 

As an aside, let us point out why the subgroup of elements that are a product of an even number of transpositions is not a proper nontrivial normal subgroup of $V$ (as with the symmetric and alternating groups): this subgroup is not proper. Indeed, every element of $V$ is a product of an even number of transpositions as the selfsimilarity of $\C$ shows that $(u \, v)$ can be rewritten as $(u0 \, v0)(u1 \, v1)$.

Bleak and Quick \cite[Theorem~1.1]{ref:BleakQuick17} demonstrated how this generating set gives a presentation for $V$ combining the corresponding presentation for the symmetric group in \eqref{eq:presentation_symmetric} with the selfsimilarity of $\C$, namely 
\begin{equation} \label{eq:presentation_v}
\< t_{u,v} \mid t_{u,v}^2, \ t_{u,w}^{t_{x,y}} = t_{u(x \, y),v(x \, y)}, \ t_{u,v} = t_{u0,v0}t_{u1,v1} \>
\end{equation}
(see \cite[(1.1)]{ref:BleakQuick17} for a full explanation of the notation used in the relations).

We will say no more about presentations, save that Bleak and Quick found a presentation for $V$ with 2 generators and 7 relations \cite[Theorem~1.3]{ref:BleakQuick17}, which they derived from another, more intuitive, presentation based on the analogy with $S_n$ which has 3 generators and 8 relations \cite[Theorem~1.2]{ref:BleakQuick17}.

As with the symmetric group, the fact that $V$ is generated by transpositions yields an easy proof of the following.

\begin{lemma} \label{lem:covering_v}
Let $U_1, \dots, U_k \subseteq \C$ be clopen subsets satisfying $U_i \cap U_{i+1} \neq \emptyset$ and $\bigcup_{i=1}^k U_i = \C$. Then $V_{[U_i]} \cong V$ for all $i$, and $V = \< V_{[U_1]}, \dots, V_{[U_k]} \>$.
\end{lemma}

With Lemma~\ref{lem:covering_v} in place, we now highlight the main ideas in the proof of Theorem~\ref{thm:donoven_harper} by way of an example (this is \cite[Example~4.1]{ref:DonovenHarper20}). We will see an alternative approach in Theorem~\ref{thm:donoven_harper_hyde}

\begin{example} \label{ex:donoven_harper}
Let $x = (00 \ \ 01) \in V$. We will construct $y \in V$ such that $\<x,y\> = V$. Let $y_1 = a_{[00]}$ and $y_2 = b_{[01]}$, where for $g \in V$ and clopen $A \subseteq \C$ we write $g_{[A]}$ for the image of $g$ under the canonical isomorphism $V \to V_{[A]}$. Let $y_3 = (00 \ \ 01 \ \ 10 \ \ 11)_{[0^310]} \cdot (0^310^3 \ \ 010^3)$ and $y_4 = (0000 \ \ 0001 \ \ \cdots \ \ 1010)_{[0^31^2]} \cdot (0^31^20^4 \ \ 1)$, and define $y = y_1y_2y_3y_4$. Note that $y_1$, $y_2$, $y_3$ and $y_4$ have coprime orders (6, 7, 5 and 11, respectively). Moreover, these elements have disjoint support, so they commute. Consequently, all four elements are suitable powers of $y$ and are, thus, contained in $\<x,y\>$. We claim that $\<x,y\> = V$. Recall that $V = \< a,b\>$, so $V_{[00]} = \<a_{[00]},b_{[00]}\> = \<y_1, y_2^x\> \leq \<x,y\>$. In addition, $V_{[01]} = (V_{[00]})^x \leq \<x,y\>$. Using appropriate elements from $V_{[00]}$ and $V_{[01]}$ we can show that $(000 \ \ 01) \in \< V_{[00]}, V_{[01]}, y_3 \> \leq \<x,y\>$ and $(000 \ \ 1) \in \< V_{[00]}, y_4 \> \leq \<x,y\>$. Therefore, $\<x,y\> \geq \< V_{[00]}, V_{[00]}^{(000 \ \ 01)}, V_{[00]}^{(000 \ \ 1)} \> = \< V_{[000 \cup 001]}, V_{[01 \cup 001]}, V_{[1 \cup 001]} \>$. Now applying Lemma~\ref{lem:covering_v} twice gives $\< x,y \> = V$.
\end{example}

\subsection{\boldmath Generalisations of $V$} \label{ss:infinite_thompson_general}

There are numerous variations on $V$, and these are the focus of this section.

The \emph{Higman--Thompson group} $V_n$, for $n \geq 2$, is an infinite finitely presented group, introduced by Higman in \cite{ref:Higman74}. There is a natural action of $V_n$ on $n$-ary Cantor space $\mathfrak{C}_n = \{0,1,\dots,n-1\}^{\mathbb{N}}$, and $V_2$ is nothing other than $V$. The derived subgroup of $V_n$ equals $V_n$ for even $n$ and has index two for odd $n$. In both cases, $V_n'$ is simple and both $V_n$ and $V_n'$ are $2$-generated \cite{ref:Mason77}. 

The \emph{Brin--Thompson group} $nV$, for $n \geq 1$, acts on $\mathfrak{C}^n$ and was defined by Brin in \cite{ref:Brin04}. The groups $V=1V, 2V, 3V, \dots$ are pairwise nonisomorphic \cite{ref:BleakLanoue10}, simple \cite{ref:Brin10} and $2$-generated \cite[Corollary~1.3]{ref:Quick19}. 

The results about generating $V$ by transpositions have analogues for $V_n$ and $nV$ (see \cite[Section~3]{ref:DonovenHarper20}), and, in \cite[Theorem~1.1]{ref:Quick19}, Quick gives a presentation for $nV$ analogous to the one for $V$ in \eqref{eq:presentation_v}. Theorem~\ref{thm:donoven_harper} extends to all of these groups too \cite[Theorems~1 \& 2]{ref:DonovenHarper20}.

\begin{theorem} \label{thm:donoven_harper_generalisation}
For all $n \geq 2$, the Higman--Thompson groups $V_n$ and $V_n'$ are $\frac{3}{2}$-generated, and for all $n \geq 1$, the Brin--Thompson group $nV$ is $\frac{3}{2}$-generated.
\end{theorem}

In particular, the groups $V_n$ when $n$ is odd give infinitely many examples of infinite $\frac{3}{2}$-generated groups that are not simple.

As we introduced them, the Higman--Thompson group $V_n'$ is a simple subgroup of $\Homeo(\C_n)$ and the Brin--Thompson group $nV$ is a simple subgroup of $\Homeo(\C^n)$. Since $\C^n$ ($n$th power of $\C$) and $\C_n$ ($n$-ary Cantor space) are both homeomorphic to $\C$, all of these groups can be viewed as subgroups of $\Homeo(\C)$. Recent work of Bleak, Elliott and Hyde \cite{ref:BleakElliottHyde}, highlights that these groups, and numerous others (such as Nekrashevych's simple groups of dynamical origin), can be viewed within one unified dynamical framework. 

A group $G \leq \Homeo(\C)$ is said to be \emph{vigorous} if for any clopen subsets $\emptyset \subsetneq B,C \subsetneq A \subseteq \C$ there exists $g \in G$ supported on $A$ such that $Bg \subseteq C$. In \cite{ref:BleakElliottHyde}, Bleak, Elliot and Hyde study vigorous groups and, among much else, prove that a perfect vigorous group $G \leq \Homeo(\C)$ is simple if and only if it is generated by its elements of \emph{small support} (namely, elements supported on a proper clopen subset of $\C$). To give a flavour of how these dynamical properties suitably capture the ideas we have seen in this section, compare the following, which is \cite[Lemma~2.18 \& Proposition~2.19]{ref:BleakElliottHyde}, with Lemma~\ref{lem:covering_v}.

\begin{lemma} \label{lem:covering_vigorous}
Let $G$ be a vigorous group that is generated by its elements of small support. Let $U_1, \dots, U_k \subseteq \C$ be clopen subsets satisfying $U_i \cap U_{i+1} \neq \emptyset$ and $\bigcup_{i=1}^k U_i = \C$. Then $G = \< G_{[U_1]}, \dots, G_{[U_k]} \>$. Moreover, if $G$ is simple, then for each $i$ the group $G_{[U_i]}$ is a simple vigorous group.
\end{lemma}

Bleak, Elliott and Hyde go on to prove that every finitely generated simple vigorous group is $2$-generated \cite[Theorem~1.12]{ref:BleakElliottHyde}. Are all such groups $\frac{3}{2}$-generated? Bleak, Donoven, Harper and Hyde \cite{ref:BleakDonovenHarperHyde} recently proved that $u(G) \geq 1$.

\begin{theorem} \label{thm:donoven_harper_hyde}
Let $G \leq \Homeo(\C)$ be a finitely generated simple vigorous group. Then there exists an element $s \in G$ of small support and order 30 such that for every nontrivial $x \in G$ there exists $y \in s^G$ such that $\< x, y \> = G$.
\end{theorem}

Theorem~\ref{thm:donoven_harper_hyde} gives $u(G) \geq 1$ for all the simple groups $G$ in Theorem~\ref{thm:donoven_harper_generalisation}. In particular, we obtain a strong version of Theorem~\ref{thm:donoven_harper} on Thompson's group $V$, improving $s(V) \geq 1$ to $u(V) \geq 1$. It is possible to obtain stronger results on the (uniform) spread of $V$ and its generalisations (and $T$, discussed below), and this is the subject of current work of the author and others (e.g. \cite{ref:BleakDonovenHarperHyde}).

\subsection{\boldmath Thompson's groups $T$ and $F$} \label{ss:infinite_thompson_t}

In this final section, we discuss generating sets of Thompson's groups $T$ and $F$. We begin with $T$, which is a simple $2$-generated group, so it is natural to study its (uniform) spread. In 2022, Bleak, Harper and Skipper \cite{ref:BleakHarperSkipper} proved $u(T) \geq 1$.

\begin{theorem} \label{thm:bleak_harper_skipper}
There exists an element $s \in T$ such that for every nontrivial $x \in T$ there exists $y \in s^T$ such that $\< x, y \> = T$.
\end{theorem}

\begin{corollary} \label{cor:bleak_harper_skipper}
Thompson's group $T$ is $\frac{3}{2}$-generated.
\end{corollary}

The element $s$ in Theorem~\ref{thm:bleak_harper_skipper} can be chosen as the one in Figure~\ref{fig:elements}. Moreover, in \cite[Proposition~3.1]{ref:BleakHarperSkipper}, it is shown that if we restrict to elements $x$ of infinite order, then we can choose $s$ to be any infinite order element, that is to say, for any two infinite order elements $x,s \in T$ there exists $g \in T$ such that $\< x, s^g \> = T$. This naturally raises the question of whether an arbitrary infinite order element can be chosen for $s$ in Theorem~\ref{thm:bleak_harper_skipper} (see \cite[Question~1]{ref:BleakHarperSkipper}).

We now turn to Thompson's group $F$. This is $2$-generated since if we write $x_0 = (00, 01, 1) \mapsto (0, 10, 11)$ and $x_1 = (0, 100, 101, 1) \mapsto (0, 10, 110, 111)$, then, by \cite[Theorem~3.4]{ref:CannonFloydParry96} for example, $F = \< x_0, x_1 \>$. Moreover, if we inductively define $x_{i+1} = x_i^{x_0}$ for all $i \geq 1$, then the elements $x_0, x_1, x_2, \dots$ witness the following well-known presentation 
\[
F = \< x_0, x_1, x_2, \dots \mid \text{$x_j^{x_i} = x_{j+1}$ for $i < j$} \>.
\]

However, $F$ is not a simple group. Considering $F$ in its natural action on $[0,1]$, the homomorphism $\pi\:F \to \Int^2$ defined as $f \mapsto (\log_2{f'(0^+)}, \log_2{f'(1^-)})$ is surjective and the kernel of $\pi$ is the derived subgroup $F'$, which is simple. Moreover, $F'$ is the unique minimal normal subgroup of $F$, so the nontrivial normal subgroups of $F$ are in bijection with normal subgroups of $F/F' = \Int^2$ (see \cite[Section~4]{ref:CannonFloydParry96} for proofs of these claims). In particular, $F$ is not $\frac{3}{2}$-generated since it has a proper noncyclic quotient. 

Now $F'$ is not $\frac{3}{2}$-generated for a different reason: it is not finitely generated. Indeed, for any nontrivial normal subgroup $N = \pi^{-1}(\< (a_0,a_1), (b_0,b_1) \>)$,  if $\{ a_0, b_0 \} = \{ 0 \}$ or $\{ a_1, b_1 \} = \{ 0 \}$, then $N$ is not finitely generated. To see this in the former case, for finitely many elements each of which acts as the identity on an interval containing $0$, there exists an interval containing $0$ on which they all act as the identity, so they generate a proper subgroup of $N$ (for the latter case, replace $0$ with $1$). However, the following recent theorem of Golan \cite[Theorem~2]{ref:GolanGen} shows that these are the only obstructions to $\frac{3}{2}$-generation.

\begin{theorem} \label{thm:golan}
Let $(a_0,a_1), (b_0,b_1) \in \Int^2$ with $\{a_0,b_0\} \neq \{0\}$ and $\{a_1,b_1\} \neq \{0\}$. Let $x \in F$ be a nontrivial element such that $\pi(x) = (a_0,a_1)$. Then there exists $y \in F$ such that $\pi(y) = (b_0,b_1)$ and $\<x,y\> = \pi^{-1}(\<(a_0,a_1),(b_0,b_1)\>$.
\end{theorem}

Theorem~\ref{thm:golan} has the following consequence, which asserts that $F$ is almost $\frac{3}{2}$-generated \cite[Theorem~1]{ref:GolanGen}.

\begin{corollary} \label{cor:golan}
Let $f \in F$ and assume that $\pi(f)$ is contained in a generating pair of $\pi(F)$. Then $f$ is contained in a generating pair of $F$.
\end{corollary}

Theorem~\ref{thm:golan} also implies that every finitely generated normal subgroup of $F$ is $2$-generated. In particular, every finite index subgroup of $F$ is $2$-generated.

\vspace{0.5\baselineskip}

\textbf{Methods. Covering lemmas and a generation criterion. } We conclude the survey by discussing how Theorems~\ref{thm:bleak_harper_skipper} and~\ref{thm:golan} are proved in \cite{ref:BleakHarperSkipper} and \cite{ref:GolanGen}, respectively. Covering lemmas (analogues of Lemma~\ref{lem:covering_symmetric}), again, play a role. For $F$ and $T$, these results are well known, see \cite[Corollary~2.6 \& Lemma~2.7]{ref:BleakHarperSkipper} for example. (We call an interval $[a,b]$ \emph{dyadic} if $a,b \in \Int[\frac{1}{2}]$.)

\begin{lemma} \label{lem:covering_f}
Let $[a_1,b_1], \dots, [a_k,b_k] \subseteq [0,1]$ be dyadic intervals satisfying $\bigcup_{i=1}^k (a_i,b_i) = (0,1)$. Then $F_{[a_i,b_i]} \cong F$ for all $i$, and $F = \< F_{[a_1,b_1]}, \dots, F_{[a_k,b_k]} \>$.
\end{lemma}

\begin{lemma} \label{lem:covering_t}
Let $[a_1,b_1], \dots, [a_k,b_k] \subseteq \S^1$ be dyadic intervals satisfying $\bigcup_{i=1}^k (a_i,b_i) = \S^1$. Then $T_{[a_i,b_i]} \cong F$ for all $i$, and $T = \< T_{[a_1,b_1]}, \dots, T_{[a_k,b_k]} \>$.
\end{lemma}

Another key ingredient is a criterion due to Golan, for which we need some further notation. Fix a subgroup $H \leq F$. An element $f \in F \leq \Homeo([0,1])$ is \emph{piecewise-$H$} if there is a finite subdivision of $[0,1]$ such that on each interval in the subdivision, $f$ coincides with an element of $H$. The closure of $H$, written ${\rm Cl}(H)$, is the subgroup of $F$ containing all elements that are piecewise-$H$. The following result combines \cite[Theorem~1.3]{ref:GolanMAMS} with \cite[Theorem~1.3]{ref:GolanMax}.

\begin{theorem} \label{thm:golan_criterion}
Let $H \leq F$. Then the following hold:
\begin{enumerate}
\item $H \geq F'$ if and only if ${\rm Cl}(H) \geq F'$ and there exist $f \in H$ and a dyadic $\omega \in (0,1)$ such that $f'(\omega^+)=2$ and $f'(\omega^-)=1$
\item $H = F$ if and only if ${\rm Cl}(H) \geq F'$ and there exist $f,g \in H$ such that $f'(0^+)=g'(1^-)=2$ and $f'(1^-)=g'(0^+)=1$.
\end{enumerate}
\end{theorem}

We now discuss the proof of Theorem~\ref{thm:bleak_harper_skipper} on $T$ given by Bleak, Harper and Skipper \cite{ref:BleakHarperSkipper}. By Lemma~\ref{lem:covering_t}, for each nontrivial $x \in T$ it suffices to find a dyadic interval $[a,b] \subseteq \S^1$ and $y \in s^T$ such that $\bigcup_{g \in \< x,y \>} (a,b)g = \S^1$ and $T_{[a,b]} \leq \< x, y \>$. If $|x|$ is infinite, a dynamical argument is used (for any infinite order element $s$), see \cite[Proposition~3.1]{ref:DonovenHarper20}. The key ingredients for finite $|x|$ are highlighted in the following example.

\begin{example} \label{ex:bleak_harper_skipper}
Let $x \in T$ be a nontrivial torsion element. We will prove that there exists $y \in s^T$ (for $s$ as in Figure~\ref{fig:elements}) such that $\< x, y \> = T$. By replacing $x$ by a power if necessary, $x$ has rotation number $\frac{1}{p}$ for prime $p$. For exposition, we only discuss the case $p \geq 5$. Since any two torsion elements of $T$ with the same rotation number are conjugate, by replacing $x$ by a conjugate if necessary, $x = (00, 01, 10, 110, \dots, 1^{p-3}0, 1^{p-2}) \mapsto (01, 10, 110, \dots, 1^{p-3}0, 1^{p-2}, 00)$. 

We claim that $T = \< x, s \>$. By Lemma~\ref{lem:covering_t}, since $\S^1 = \bigcup_{i \in \Int} (0,\frac{7}{8})x^i$, it suffices to prove that $T_{[0,\frac{7}{8}]} \leq \< x, s \>$. Indeed, we claim that $T_{[0,\frac{7}{8}]} = \<y_0,y_1\>$ for $y_0 = s$ and $y_1 = s^x$. Defining $t\:(0,\frac{7}{8}) \to (0,1)$ as $\omega t = \omega$ if $\omega \leq \frac{3}{4}$ and $\omega t = 2\omega-\frac{3}{4}$ if $\omega > \frac{3}{4}$, it suffices to prove that $\< y_0^t, y_1^t \> = (T_{[0,\frac{7}{8}]})^t = F$. 

To do this, we apply Theorem~\ref{thm:golan_criterion}(ii). To verify the second condition, choose $f = y_0^t$ and $g = (y_1^t)^{-1}$, so $f'(0^+)=g'(1^-)=2$ and $f'(1^-)=g'(0^+)=1$. It remains to prove that ${\rm Cl}(\< y_0^t, y_1^t \>) \geq F'$. Here we apply another criterion: for $g_1, \dots, g_k \in F$ we have $\< g_1, \dots, g_k \> \geq F'$ if and only if the Stallings $2$-core of $\<g_1, \dots, g_k\>$ equals the Stallings $2$-core of $F$ \cite[Lemma~7.1 \& Remark~7.2]{ref:GolanMAMS}. The \emph{Stallings $2$-core} is a directed graph associated to a diagram group introduced by Guba and Sapir \cite{ref:GubaSapir97}. Given elements $g_1, \dots, g_k \in F$ represented as tree pairs, there is a short combinatorial algorithm to find the Stallings 2-core of $\< g_1, \dots, g_k \>$, and it is straightforward to compute the Stallings 2-core of $\< y_0^t, y_1^t \>$ and note that it is the Stallings 2-core of $F$ (see the proof of \cite[Proposition~3.2]{ref:BleakHarperSkipper}). Therefore, $F = \< y_0^t, y_1^t \>$, completing the proof that $T = \< x, s \>$.
\end{example}

We conclude by briefly outlining the proof of Theorem~\ref{thm:golan} on $F$ given by Golan \cite{ref:GolanGen}, which uses similar methods to those in \cite{ref:BleakHarperSkipper} on $T$ and \cite{ref:DonovenHarper20} on $V$. Let $(a_0,a_1), (b_0,b_1) \in \Int^2$ with $\{a_0,b_0\} \neq \{0\}$ and $\{a_1,b_1\} \neq \{0\}$, and let $x \in F \setminus 1$ with $\pi(x) = (a_0,a_1)$. Observe that it suffices to find an element $y$ such that $\pi(y) = (b_0,b_1)$ and $\< x,y \> \geq F'$. In \cite{ref:GolanGen}, an explicit choice of $y$, based on $x$, is given and the condition $\< x,y \> \geq F'$ is verified via Theorem~\ref{thm:golan_criterion}(i).

\end{document}